\documentclass[leqno,12pt,a4paper]{article}
\usepackage[dvips]{graphicx}
\usepackage{amssymb}
\usepackage{theorem}
\usepackage{color}
\title
{On the global existence of generalized rotational hypersurfaces with prescribed mean curvature in the Euclidean 
spaces.  I}
\author{Katsuei Kenmotsu and Takeyuki Nagasawa\footnote{The second named author is partly supported by Grant-in-Aid for scientific Research (A)(No.22244010-01), Japan Society for the  Promotion of Science.}}
\date{}
\newcommand{\qed}{{\unskip\nobreak\hfil\penalty50\quad\null\nobreak\hfil
	$\square$\parfillskip0pt\finalhyphendemerits0\par\medskip}}
\newcommand{\pref}[1]{(\ref{#1})}

\newtheorem{thm}{Theorem}[section]
\newtheorem{prop}{Proposition}[section]
\newtheorem{lem}{Lemma}[section]
\newtheorem{cor}{Corollary}[section]

\theorembodyfont{\rmfamily}
\newtheorem{rem}{Remark}[section]
\newtheorem{proof}{\normalfont\itshape Proof.}

\begin{document}
\maketitle
\begin{abstract}
We prove that any piece of a rotational hypersurface  with prescribed mean curvature function in a Euclidean space can 
be uniquely extended infinitely, which generalizes the results by Euler and Delaunay for surfaces of revolution with constant mean curvautre. Next, we prove the same kind of theorem for generalized rotational hypersurfaces of $O(l+1)\times O(m+1)$-type. The key lemmas in this paper show  the existence of solutions for singular initial value problems which arise from the analysis of ordinary differential equations of generating curves of those hypersurfaces. 

\begin{flushleft}
\end{flushleft}
 
\end{abstract}
\section{Introduction}
Surfaces of revolution with constant mean curvature in the Euclidean three-space $\mathbb{R}^{3}$ can be uniquely extended infinitely by Delaunay's rolling construction method \cite{del}.
This result was generalized to higher dimensions by Hsiang and Yu \cite{hsyu}.
Another proof of Delaunay--Hsiang--Yu's theorem was given in Dorfmeister and Kenmotsu \cite{doke1}.
\par
There is a different approach to generalizing Delaunay's theorem.
In Kenmotsu \cite{ken2},
one of the authors of this paper extended the periodicity property of surfaces of revolution with constant mean curvature to surfaces of revolution with non-constant mean curvature.
To consider this extension, one must first prove the global existence of such surfaces for a given periodic mean curvature defined on the whole line $\mathbb{R}$.
\par
Kenmotsu \cite{ken1} showed that for a given continuous function $H(s)$ on $\mathbb{R}$,
there exists a global surface of revolution such that the mean curvature is $H(s)$; this extends Delaunay's result to the case of non-constant mean curvature.
Later,
in 2009,
Dorfmeister and Kenmotsu \cite{doke2} extended this result to higher dimensions under the condition that the mean curvature function is real analytic.
\par
The purpose of this paper is two-fold.
First,
we extend the results of Kenmotsu \cite{ken1} and Dorfmeister and Kenmotsu \cite{doke2} to rotational hypersurfaces with the mean curvature function being continuous;
next,
we shall prove the same kind of  existence theorem and study some properties  for a class of generalized rotational hypersurfaces.
\par
A generalized rotational hypersurface $M$ in the $n$-dimensional Euclidean space $\mathbb{R}^{n}$ is defined by a compact Lie group $G$ and its representation to $\mathbb{R}^{n}$ \cite{hsjdg},
i.e.,
$M$ is invariant under an isometric transformation group $(G,\mathbb{R}^{n})$ with  codimension  two principal orbit 
type.
Such transformation groups $(G,\mathbb{R}^{n})$ have already been classified in \cite{hsla};
for $n=3$,
we have only $G=O(2)$,
and for $n\geqq 4$, there are 14 Lie groups.
\par
In this paper, we study  generalized rotational hypersurfaces of  $O(n-1)$-type with prescribed 
mean curvature function, which are equivalent to the notion of rotational hypersurfaces,
and of  $O(l+1) \times O(m+1)$-type,
and we prove the global existence of such hypersurfaces under the condition that a given mean curvature is continuous.
The main results of this paper are stated in Theorems \ref{thm.2.1} and \ref{thm.4.1}. Proofs of these theorems use Banach's fixed-point theorem on an appropriate class of functions, and, thus,
the main task in this paper is to estimate these terms appearing in differential equations of generating curves.  In section 3, we shall study some properties of generalized rotational hypersurfaces of  $O(n-1)$-type by applying Theorem \ref{thm.2.1}.

The analysis and the geometry of generalized rotational hypersurfaces with prescribed mean curvature function of other 
types will be studied in succeeding papers.

\par
\section{Global existence of generalized rotational hypersurfaces of $O(n-1)$-type}
In this section,
we prove the global existence of  generalized rotational hypersurfaces of $O(n-1)$-type with prescribed mean curvature function.
Let $(x(s),y(s))$,
$ y(s) > 0 $ $( s \in \mathbb{R})$ be a plane curve parametrized by  arc length in the Euclidean two-plane
$\mathbb{R}^2$ defined by $ x_{3} = \cdots = x_{n}=0 $,
where $x_{i} $,
$ i= 1 $,
$ 2 $,
$ \cdots $,
$ n $,
are the standard coordinates of 
$\mathbb{R}^n$.
A generalized rotational hypersurface $M$ of $O(n-1)$-type is  defined by
\begin{equation}
	M = \left\{ (x(s),y(s)S^{n-2})\in \mathbb{R}^n \ | \ s \in \mathbb{R} \right\},
	\label{Eq.(1)}
\end{equation}
where $S^{n-2}$ is an $(n-2)$-dimensional unit sphere with center origin in the $(n-1)$-dimensional Euclidean space defined by $x_{1}=0$.
The plane curve $(x(s),y(s))$ is called the generating curve of $M$. 
\par
The mean curvature of $M$ is a function of one variable $s$, denoted by $H=H(s)$, and satisfies
\begin{equation}
	( n - 1 ) H(s)
	=
	( n - 2 ) \frac {x^\prime (s)} {y(s)}
	+
	x^{\prime\prime}(s) y^\prime(s) - x^\prime(s) y^{\prime\prime}(s) ,
	\quad   \mbox{for}  \quad
	s \in \mathbb{R} .
	\label{Eq.(2)}
\end{equation}
Component functions of the generating curve satisfy $y(s) >0$ and
\begin{equation}
	x^\prime (s)^2 + y^\prime (s)^2 = 1, \quad \mbox{for} \quad  s \in \mathbb{R},
	\label{Eq.(3)}
\end{equation}
because the parameter $s$ is arc length.   
\par
Conversely, given a continuous function $H(s)$ defined on the whole line $\mathbb{R}$,
\pref{Eq.(2)} and \pref{Eq.(3)} are a system of ordinary differential equations for $x(s)$ and $y(s)$.
\par
Let us fix an $s_{0} \in \mathbb{R}$.
Given any $c>0$, and any real numbers $c^\prime$,
$d^\prime$ satisfying ${c^\prime}^2 + {d^\prime}^2 = 1$,
 the usual existence theorem of ordinary differential equations implies that there exists a local solution curve 
$(x(s),y(s))$ on $ \mathbb{R}$,
of the system \pref{Eq.(2)} and \pref{Eq.(3)} with the initial conditions $x(s_{0})=0$,
$y(s_{0})=c$,
$ x^\prime(s_{0})=c^\prime$,
$ y^\prime(s_{0})=d^\prime$.

When we extend the definition of the domain of these component functions to $\mathbb{R}$,
a problem happens at the point that $y(s)$ passes through $x$-axis at some finite $s$.
By Dorfmeister and Kenmotsu \cite[Proposition 3.2, p.706]{doke2}, 
we have
\begin{prop}
Suppose that $\lim_{s \rightarrow b}y(s)=0$ for some $b \in \mathbb{R}$.
Then,
there exists the 
limit of $x^\prime(s)$ as $s \rightarrow b$,
and $\lim_{s \rightarrow b}x^\prime(s)=0$.
\label{prop.2.1}
\end{prop}
\begin{rem}
The proof in $\cite{doke2}$ requires $ y^\prime (s) \ne 0 $ near $ b $,
which was not shown there.
We,
however,
can show this in a similar manner to Lemma $3.1$ of this paper.
\end{rem}

We would like to show that the solution can be extended beyond $ b $.  Without loss of generality,
we may assume $ b = 0 $, because the system of \pref{Eq.(2)} and \pref{Eq.(3)} is invariant under the translation 
parallel to $x$-axis.
It follows from Proposition \ref{prop.2.1} that the mapping from $ y $ to $ s $ is one-to-one near $ 0 $,
and therefore the inverse function $ s = s(y) $ exists.
Now we rewrite our equation considering $ y $ as the independent variable.
To do this,
let us pose
$q = \frac{x^\prime}{y^\prime}$.
Then we have,
$$
	\frac{dq}{dy}
	=
	\frac 1 {{y^\prime}^{3}}
	\left\{
	( n - 1 ) H(s) - ( n - 2 ) \frac {x^\prime} {y}
	\right\}.
$$
By \pref{Eq.(3)} and $ x^\prime (0) = 0 $, we have $ y^\prime (0) = \pm 1 $. Hence, 
$ y^\prime (s)$ does not vanish, and $ \mbox{\rm sgn } y^\prime (s) = \mbox{\rm sgn } y^\prime (0) \ne 0 $ on a neighborhood of $s=0$.
From $ ( 1 + q^2 ) {y^\prime}^2 = 1 $,
it follows that
\[
	y^\prime (s)^{3}
	= \left( \mbox{\rm sgn } y^\prime (s) \right)
	\left\{ \left( y^\prime (s) \right)^2 \right\}^{ \frac 32 }
	= \left( \mbox{\rm sgn } y^\prime (0) \right)
	\left( 1 + q(s)^2 \right)^{ -\frac 32 }.
\]
Setting
$ \tilde H(y) = \left( \mbox{\rm sgn } y^\prime (0) \right) H( s(y) ) $,
we obtain
\begin{eqnarray*}
	\frac {dx} {dy}
	& \!\!\! = & \!\!\!
	q ,
	\\
	y \frac {dq} {dy}
	& \!\!\! = & \!\!\!
	- ( n - 2 ) q - ( n - 2) q^3
	+ ( n - 1 ) \tilde{H}(y) y ( 1 + q^2 )^{\frac 32} ,
\end{eqnarray*}
where $ x $ and $ q $ are unknown functions of $ y $.
\par
Let us consider the singular initial value problem
\begin{equation}
	\quad
	\left\{
	\begin{array}{l}
	\displaystyle{
	y \frac { dq } { dy }
	=
	- ( n - 2 ) q - ( n - 2 ) q^3
	+ ( n - 1 ) \tilde H (y) y \left( 1 + q^2 \right)^{ \frac 32 },
	\  \mbox{for} \ y > 0
	},
	\\
	q(0) = 0.
	\end{array}
	\right.
	\label{DiffEq(8)}
\end{equation}
Multiplying $ y^{ n - 3 } $ to the first equation of \pref{DiffEq(8)},
we see
\[
	\frac d { dy } \left( y^{ n-2 } q \right)
	=
	\left\{- ( n - 2 ) q^3
	+
	( n - 1 ) \tilde H (y) y \left( 1 + q^2 \right)^{ 3/2 }\right\} y^{ n - 3 } .
\]
Let us integrate the above equation on $[0,y]$.
Then,
by $ \left. y^{ n-2 } q(y) \right|_{ y=0 } = 0$,  
we have 
\[
	y^{ n-2 } q(y)
	=
	\int_0^y
	\left\{
	- ( n - 2 ) q ( \eta )^3
	+ ( n - 1 ) \tilde H ( \eta ) \eta \left( 1 + q( \eta )^2 \right)^{ 3/2 }
	\right\} \eta^{ n - 3 }
	d \eta ,
\]
which leads to the following integral equation
\begin{equation}
	q(y) = \Phi(q)(y),
	\label{IntEq}
\end{equation}
where we set
$$
	\Phi(q)(y)
	=
	y^{ 2-n }
	\int_0^y
	\left\{
	- ( n - 2 ) q ( \eta )^3
	+ ( n - 1 ) \tilde H ( \eta ) \eta \left( 1 + q( \eta )^2 \right)^{ 3/2 }
	\right\}
	\eta^{ n - 3 }
	d \eta .
$$
We shall find a fixed point of the mapping $\Phi$ in an appropriate class of functions.
To do it,
let us define a function space and its subclass by
\[
	X_{Y}
	=
	\left\{
	q \in C( 0 , Y ] \, | \, \| q \| _X < \infty
	\right\} ,
	\quad
	X_{Y,M}
	=
	\{ q \in X_Y \, | \, \| q \|_X \leqq M \} ,
\]
where $Y$ and $M$ are positive constants and
\[
	\| q \|_X
	=
	\sup_{ y \in ( 0 , Y ] } \left| \frac { q(y) } y \right| .
\]
$ X_{Y} $ is a Banach space with the norm $ \| \cdot \|_X $.
When $ q \in X_{Y,M} $,
we note that $ | q(y) | \leqq \| q \|_X | y | \to 0 $ as $ y \to +0 $.
Setting $ q(0) = 0 $,
$ q $ is an element of $ C [ 0 , Y ] $.
\begin{prop}
\begin{itemize}
\item[{\rm (i)}]
Suppose that $ \tilde H $ is bounded on $[0,Y]$.
For sufficiently large $ M $,
and small $ Y $, 
there exists a unique solution $q$ of the integral equation {\rm \pref{IntEq}} on $ X_{Y,M} $.
\item[{\rm (ii)}]
Suppose that $ \tilde H $ is bounded and continuous on $[0,Y]$.
Then the function $q$ obtained in the 
above {\rm (i)} is a solution of the initial value problem {\rm \pref{DiffEq(8)}}.
\end{itemize}
\label{prop.2.2}
\end{prop}
\begin{proof}
Suppose that $ \tilde H $ is bounded on $[0,Y]$.
From now on,
$C$ denotes a positive constant which depends on $n$ and 
$ \sup_y | \tilde H(y) | $,
but does not depend on $M$ and $Y$.
\par
First,
we show that $ \Phi $ is a mapping from $ X_{Y,M} $ to $ X_{Y,M} $ for some large $M$ and  small $Y$.
Take $ q \in X_{Y,M} $. Then we have $ \Phi (q) \in C( 0 , Y ] $.
We show $ \| \Phi (q) \|_X \leqq M $ as follows:
By using $ | q( \eta ) | \leqq M \eta $ for $ \eta \in ( 0 , Y ] $,
we see
\[
	\begin{array}{l}
	\displaystyle{
	\left|
	\left\{
	- ( n - 2 ) q ( \eta )^3
	+ ( n - 1 ) \tilde H ( \eta ) \eta \left( 1 + q( \eta )^2 \right)^{ 3/2 }
	\right\}
	\eta^{ n - 3 }
	\right|}
	\\
	\quad
	\leqq
	\displaystyle{
	C \left\{
	| q( \eta ) |^3	+ \eta \left( 1 + | q( \eta ) |^3 \right)
	\right\}
	\eta^{ n - 3 }
	\leqq
	C \left(
	M^3 \eta^n + \eta^{ n-2 } + M^3 \eta^{ n+1 }
	\right)
	}.
\end{array}
\]
Hence,
\[
	\left|\frac { \Phi (q) (y) } y \right|
	\leqq
	C y^{ 1-n }
	\int_0^y
	\left( M^3 \eta^n + \eta^{ n-2 } + M^3 \eta^{ n+1 }\right)
	d \eta
	\leqq
	C \left( M^3 y^2 + 1 + M^3 y^3 \right).
\]
Take $ M $ and $ Y $ such that
\[
	C \left( M^3 Y^2 + 1 + M^3 Y^3 \right) \leqq M .
\]
Then,
$ \| \Phi (q) \|_X \leqq M $.
\par
Next,
if necessary,
taking $Y$ much smaller,
we show that $ \Phi $ is a contraction mapping from $ X_{Y,M} $ to itself. 
Take $ q_1 $,
$ q_2 \in X_{Y,M} $.
Then,
\[
	\Phi ( q_1 ) (y) - \Phi ( q_2 ) (y)
	=
	y^{ 2-n }
	\int_0^y
	\left( \phi ( q_1 ( \eta ) ) - \phi ( q_2 ( \eta ) ) \right)
	\eta^{ n-3 } d \eta ,
\]
where
\[
	\phi ( q )
	=
	- ( n - 2 ) q^3
	+ ( n - 1 ) \tilde H( \eta ) \eta \left( 1 + q^2 \right)^{ 3/2 } .
\]
By the mean value theorem, 
there is a $ q_\ast $ between $ q_1 $ and  $ q_2 $,
such that
\[
	\phi ( q_1 ) - \phi ( q_2 )
	=
	\phi^\prime ( q_\ast ) ( q_1 - q_2 ) .
\]
We note that
\[
	| q_\ast ( \eta ) |
	\leqq
	\max\{ | q_1 ( \eta )| , | q_2 ( \eta ) | \}
	\leqq
	M \eta ,
\]
and also, since we have
$ \phi^\prime (q) = - 3 ( n - 2 ) q^2 + 3 ( n - 1 ) \tilde H ( \eta ) \eta \left( 1 + q^2 \right)^{ 1/2 } q $,
\[
	\begin{array}{rl}
	| \phi^\prime ( q_\ast ( \eta ) ) |
	\leqq
	& \!\!\!
	\displaystyle{
	C \left\{
	| q_\ast ( \eta ) |^2
	+ \eta \left( 1 + | q_\ast ( \eta ) |^2 \right)^{ 1/2 } | q_\ast ( \eta ) |
	\right\}
	}
	\\
	\leqq
	& \!\!\!
	\displaystyle{
	C \left(
	M^2 \eta^2 + \eta \left( 1 + M \eta \right) M \eta
	\right)
	=
	C \left( M^2 \eta^2 + M \eta^2 + M^2 \eta^3 \right)
	} .
	\end{array}
\]
Combining these two estimates,
we obtain
\[
	\begin{array}{rl}
	\displaystyle{
	\left|
	\frac { \Phi ( q_1 ) (y) - \Phi ( q_2 ) (y) } y
	\right|
	}
	\leqq
	& \!\!\!
	\displaystyle{
	C y^{1-n}
	\int_0^y
	\left( M^2 \eta^2 + M \eta^2 + M^2 \eta^3 \right)
	| q_1 ( \eta ) - q_2 ( \eta ) |
	\eta^{ n-3 }
	d \eta
	}
	\\
	\leqq
	& \!\!\!
	\displaystyle{
	C y^{1-n} \| q_1 - q_2 \|_X
	\int_0^y
	\left( M^2 \eta^2 + M \eta^2 + M^2 \eta^3 \right)
	\eta^{ n-2 }
	d \eta}
	\\
	\leqq
	& \!\!\!
	\displaystyle{
	C \left( M^2 y^2 + M y^2 + M^2 y^3 \right)\| q_1 - q_2 \|_X
	} .
\end{array}
\]
Consequently,
by choosing $Y$ so that $ C \left( M^2 Y^2 + M Y^2 + M^2 Y^3 \right) < 1 $,
Banach's fixed point theorem implies that there exists a unique fixed point $q$ of  $ \Phi $ on $X_{Y,M}$ which satisfies {\rm \pref{IntEq}}.
\par
To prove (ii),
let $ q $ be a solution of {\rm \pref{IntEq}}.
Then we have
\[
	y^{ n-2 } q(y)
	=
	\int_0^y \phi ( q( \eta ) ) \eta^{ n-3 } d \eta .
\]
Since $ q \in C[ 0 , Y ] $,
if $ \tilde H $ is continuous,
then $ \phi ( q( \eta) ) $ is also continuous on $ [0,Y] $.
Thus, 
the right hand side above is differentiable,
which implies that $q$ is also differentiable.
By taking the derivation of the formula above,
we have the first equation of {\rm \pref{DiffEq(8)}}.
By $ q \in X_{Y,M} $,
we have $q(0) = 0$.
\qed
\end{proof}
\par
By replacing \cite[Proposition 3.3]{doke2} to Proposition \ref{prop.2.2} of this paper, the  proof of \cite[Theorem 3.4]{doke2} can be used to prove the  following 
\begin{thm}
Let $H(s)$ be a continuous function on $\mathbb{R}$,
and fix an $s_{0} \in \mathbb{R}$.
Then,
for any $c>0$, and any real numbers $c^\prime$,
$d^\prime$ satisfying ${c^\prime}^2 + {d^\prime}^2 = 1$,
there exists a global solution curve $(x(s),y(s))$, for $ s \in \mathbb{R}$,
of the system {\rm \pref{Eq.(2)}} and {\rm \pref{Eq.(3)}} with the initial conditions $x(s_{0})=0$,
$y(s_{0})=c$,
$ x^\prime(s_{0})=c^\prime$,
$ y^\prime(s_{0})=d^\prime$.
\label{thm.2.1}
\end{thm}
\begin{rem}
The global solution curve in Theorem \ref{thm.2.1} is extended smoothly in $(x,y)$-plane with $y<0$ when it touches at $x$-axis. The global generating curve is obtained by the reflection of the solution curve with respect to the $ x $-axis.
\end{rem}
\section{Properties of generalized rotational hypersurfaces of $O(n-1)$-type}
In this section we shall apply Theorem \ref{thm.2.1} to study some properties of generalized rotational hypersurfaces of $O(n-1)$-type with non-constant mean curvature function $H(s)$.
We note that the interesting properties of those hypersurfaces with constant mean curvature were studied by Hsiang and Yu \cite{hsyu} in 1981.
\par
Given any continuous $H(s)$ on $\mathbb{R}$ and any $c>0$,
by Theorem \ref{thm.2.1} there exists uniquely a global solution curve $(x_{c}(s),y_{c}(s)), $ for $ s \in 
\mathbb{R}, $ of the system \pref{Eq.(2)} and \pref{Eq.(3)} with the initial conditions $x_{c}(0)=0$, $y_{c}(0)=c>0 $,
$ x_{c}'(0)=1 $,
$ y_{c}'(0)=0$.
Let $\Gamma_{c}=(x_{c}(s),|y_{c}(s)|)$. The curve $\Gamma_{c}$ is the generating curve of a   
generalized rotational hypersurface,
say $M_{c}$, of $O(n-1)$-type. 
 $\Gamma_{c}$ has possibly the singularity for the induced metric at $y_{c}(s)=0$, because the first 
 fundamental form of $M_{c}$ is the direct product  of $ds^{2}$ and the coformal metric of $S^{n-2}$ 
 with the conformal factor $y(s)^{2}$.
We prove
\begin{thm}  Let $H(s)$ be an absolutely continuous  function on $\mathbb{R}$ with  $H(0) >0$ such that  $H'(s) \geqq 0$ a.\ e.\ $s \in (0, \infty)$,
and $H'(s) \leqq 0 $ a.\ e.\ $s \in (-\infty,0)$.
Then,  for any $c$ satisfying $c>1/H(0)$, $y_{c}(s)$  is positive on 
 $\mathbb{R}$, and  $M_{c}$ is an immersed hypersurface in $\mathbb{R}^{n}$.
\label{thm.3.1}
\end{thm}
\begin{proof}
By contraries, suppose that there exists an $s_{0} \in \mathbb{R}^{+}$ such that $y_{c}(s_{0})=0$ and $y_{c}(s) >0 \ \mbox{on} \ s \in [0,s_{0})$.
By \cite[(3.2)]{doke2} and the initial conditions of $\Gamma_{c}$, we have
  $$
  y_{c}^{n-2}(s)x_{c}'(s) = (n-1)\int_{0}^{s}H(t)y_{c}^{n-2}(t)y_{c}'(t) \, dt + c^{n-2},  \ \mbox{for} \ s 
  \in [0 , s_0].
  $$
   Since the left hand side of the above formula is zero at  $s=s_{0}$  by the assumption and 
   \pref{Eq.(3)},
   we have
\begin{eqnarray*}
 c^{n-2} \!\!\! & =
 & \!\!\! -(n-1) \int_{0}^{s_{0}}H(t)y_{c}^{n-2}(t)y_{c}'(t) \, dt \\
 & =
 & \!\!\! - \left[ H(t) y_{c}^{n-1}(t) \right]_{0}^{s_{0}} 
 + \int_{0}^{s_{0}}H'(t)y_{c}^{n-1}(t) \, dt  \\
 & \geqq
 & \!\!\! H(0)c^{n-1}.
 \end{eqnarray*}

This contradicts the assumption of $c$. If $s_{0} <0$, then we have the same contradiction,  proving 
Theorem \ref{thm.3.1}.
 \qed
\end{proof}
\par
\begin{rem}  We have also the following:
Let $H(s)$ be an absolutely continuous  function on $\mathbb{R}$ with $H(0) >0$ such that $ H'(s) \leqq 0$ a.\ e.\ $ s \in (0, \infty)$, and $H'(s) \geqq 0$ a.\ e.\ $s \in (-\infty,0)$. Then,
for any $c$ satisfying $0<c<1/H(0)$,
$y_{c}(s)$ is positive on $\mathbb{R}$,
and $M_{c}$ is an immersed hypersurface in $\mathbb{R}^{n}$. 
\label{rem.3.1}
\end{rem}
\par
Theorem \ref{thm.3.1} and  Remark \ref{rem.3.1} extend the results for constant mean curvature case \cite{ken1,hsyu}.
\par
To continue the study of those hypersurfaces with non-constant mean curvature $H(s)$, we make the asymptotic analysis of $\Gamma_{c}$ when $c \rightarrow  \infty$. To do it, let $\Gamma_{\infty}$ be the planar curve parametrized by  arc length such that the curvature is $-(n-1)H(s)$. Actually, a planar curve parametrized by arc length is determined by its curvature only up to a rigid motion.  $\Gamma_{\infty}$ is defined on $\mathbb{R}$ by the fundamental theorem of curve theory.  In fact, we have
\begin{equation}
\Gamma_{\infty} = \left(\int_{0}^{s}\cos \eta(u)du, -\int_{0}^{s}\sin \eta(u)du \right), \ \mbox{for} \  s \in \mathbb{R},
\label{Eq.(6)}
\end{equation}
where we set $\eta(u) = (n-1)\int_{0}^{u}H(t)dt$. Set $F_{c}(s)= y_{c}(s)y_{c}'(s)/c$ and $G_{c}(s) = y_{c}(s)x_{c}'(s)/c$. Then, 
\cite[Lemma 4.1]{doke2} implies
$$
\left\{
	\begin{array}{rlrl}
	F_{c}' = & \!\!\!
	\displaystyle{
	-(n-1)H(s) G_{c}
	+ \frac{1}{c} \left\{ 1 + \frac{(n-3)G_{c}^{2}}{F_{c}^{2} + G_{c}^{2}} \right\}
	}
	,
	&
	F_{c}(0) = & \!\!\! 0
	,
	\\
	G_{c}' = & \!\!\!
	\displaystyle{
	(n-1)H(s) F_{c}
	- \frac{1}{c} \left\{ \frac{(n-3)F_{c}G_{c}}{F_{c}^{2} + G_{c}^{2}} \right\}
	}
	,
	& G_{c}(0) = & \!\!\! 1
	,
	\end{array}
	\right.
$$
for $ s \in \mathbb{R} $.
We note that the functions $G_{c}^2/(F_{c}^2+G_{c}^2)$ and $F_{c}G_{c}/(F_{c}^2+G_{c}^2)$ are globally defined  and differentiable on $\mathbb{R}$. Let $(F_{\infty}, G_{\infty})$ be the unique 
solution to the system
$$
\left\{
\begin{array}{rlrl}
F_{\infty}'= & \!\!\!
-(n-1)H(s) G_{\infty},
&
F_{\infty}(0) = & \!\!\! 0 ,
\\
G_{\infty}'= & \!\!\! (n-1)H(s) F_{\infty},
&
G_{\infty}(0)= & \!\!\! 1.
\end{array}
\right.
$$
This is integrated as  
$
F_{\infty}(s)=-\sin \eta(s), \ \mbox{and} \ G_{\infty}(s)= \cos \eta(s).
$
With these conventions, we prove
\begin{thm} It holds that
$$
\lim_{c\rightarrow \infty}F_{c}(s) =F_{\infty}(s), \ \lim_{c\rightarrow \infty}G_{c}(s) =G_{\infty}(s),
$$
compactly uniformely with respect to $s \in \mathbb{R}$.
\label{thm.3.2}
\end{thm}
\begin{proof}
Let us define $\tilde{F_{c}}$ and $\tilde{G_{c}}$ by $F_{c}=F_{\infty} + \tilde{F_{c}}$ and $G_{c}=G_{\infty} 
+ \tilde{G_{c}}$, respectively.
Then we get
$$
\left\{
	\begin{array}{rlrl}
\tilde{F_{c}}' = & \!\!\!
\displaystyle{
-(n-1)H(s)\tilde{G_{c}}
+ \frac{1}{c}\left\{ 1 + \frac{(n-3)G_{c}^{2}}{F_{c}^{2} + G_{c}^{2}} \right\}
}
,
&
\tilde{F_{c}}(0)= & \!\!\! 0, \\
\tilde{G_{c}}' = & \!\!\!
\displaystyle{
(n-1)H(s)\tilde{F_{c}}
- \frac{1}{c}\left\{ \frac{(n-3)F_{c}G_{c}}{F_{c}^{2} + G_{c}^{2}} \right\}
}
,
& \tilde{G_{c}}(0)= & \!\!\! 0.
	\end{array}
	\right.
$$
This implies
\begin{eqnarray*}
\frac{d}{ds} \left( \tilde{F_{c}}^2 + \tilde{G_{c}}^2 \right)
\!\!\! &=& \!\!\!
\frac{2}{c}\left\{\left( 1 + \frac{(n-3)G_{c}^{2}}{F_{c}^{2} + G_{c}^{2}}\right)\tilde{F_{c}} - \frac{(n-3) F_c G_c }{F_{c}^{2} + G_{c}^{2}}\tilde{G_{c}} \right\}  \\
 & \leqq & \!\!\!
 \frac{2 \sqrt 2 ( n - 2 ) }{c}\sqrt{\tilde{F_{c}}^2 + \tilde{G_{c}}^2}
 ,
 \end{eqnarray*}
where $ K $ is a constant independent of $ c $ and $ s $.
 Consequently, we have
 $$
 \tilde{F_{c}}^2 + \tilde{G_{c}}^2 \leqq \frac{ 2 ( n - 2 )^2 s^2 }{ c^2 }  \ \ \mbox{for} \ s \geq 0.$$
 The same estimate for $s<0$ follows from
 $$
 \frac{d}{ds} \left( \tilde{F_{c}}^2 + \tilde{G_{c}}^2 \right) 
 \geqq -\frac{2 \sqrt 2 ( n - 2 ) }{c}  \sqrt{\tilde{F_{c}}^2 + \tilde{G_{c}}^2},
 $$
 proving Theorem \ref{thm.3.2}.
 \qed
\end{proof}
\par
In view of $ F_{c}(s)^2 + G_{c}(s)^2 = y_{c}(s)^{2}/c^2 $,  $\Gamma_{c}$ has the following expression: for $c>0,\ s \in \mathbb{R}$,
$$
 \Gamma_{c}=\left( \int_{0}^{s}\frac{G_{c}(u)}{\sqrt{F_{c}(u)^2 + G_{c}(u)^2 }} \, du, 
 \int_{0}^{s}\frac{F_{c}(u)}{\sqrt{F_{c}(u)^2 + G_{c}(u)^2 }} \, du + c \right).
$$
 As the geometric application of Theorem 3.2 and the above formula, we have
 \begin{cor}
  In the limit $c \to \infty$,  the curves $\Gamma_{c} - (0,c)$ tend to $\Gamma_{\infty}$.
\label{cor.3.1}
 \end{cor}
 \begin{proof}
 It follows from $\Gamma'_{c} \to 
 \Gamma_{\infty}' =  (\cos \eta(s),-\sin \eta(s))$ as $c \to \infty$, proving Corollary.
 \qed
\end{proof}
\par
Now, we shall derive the asymptotic expansion formulas of $F_{c}$ and $G_{c}$ which are applied to study periodic generating curves.
 Set
\[
	U(s) =
	\left( \begin{array}{rr}
	\cos \eta (s) & - \sin \eta (s) \\
	\sin \eta (s) & \cos \eta (s)
	\end{array} \right) .
\]
Since
\begin{eqnarray*}
	U^\prime(s) U(s)^{-1}  
	=
	( n - 1 ) H(s) 
	\left( \begin{array}{rr}
	0 & -1 \\
	1 & 0
	\end{array} \right)
	,
\end{eqnarray*}
we have
\[
	\begin{array}{rl}
	\displaystyle{
	\left\{ U(s)^{-1}
	\left( \begin{array}{c} F_c(s) \\ G_c(s) \end{array} \right)
	\right\}^\prime
	}
	= & \!\!\!
	\displaystyle{
	- U(s)^{-1} U(s)^\prime U^{-1}(s)
	\left( \begin{array}{c} F_c(s) \\ G_c(s) \end{array} \right)
	+
	U(s)^{-1}
	\left( \begin{array}{c} F_c^\prime(s) \\ G_c^\prime(s) \end{array} \right)
	}
	\\
	= & \!\!\!
	\displaystyle{
	U(s)^{-1}
	\left\{
	- ( n - 1 ) H(s)
	\left( \begin{array}{rr}
	0 & -1 \\
	1 & 0
	\end{array} \right)
	\left( \begin{array}{c} F_c(s) \\ G_c(s) \end{array} \right)
	+
	\left( \begin{array}{c} F_c^\prime(s) \\ G_c^\prime(s) \end{array} \right)
	\right\}
	}
	\\
	= & \!\!\!
	\displaystyle{
	\frac 1c
	U(s)^{-1}
	\left(
	\begin{array}{c}
	1 + ( n - 3 ) P_c(s) \\ ( n - 3 ) Q_c(s)
	\end{array} \right)
	}
	,
	\end{array}
\]
where
\[
	P_c(s) = \frac { G_c(s)^2 } { F_c(s)^2 + G_c(s)^2 } ,
	\quad
	Q_c(s) = - \frac { F_c(s) G_c(s) } { F_c(s)^2 + G_c(s)^2 } .
\]
Integrating this, we obtain
\begin{equation}
	\begin{array}{rl}
	\displaystyle{
	\left( \begin{array}{c} F_c (s) \\ G_c (s) \end{array} \right)
	}
	= & \!\!\!
	\displaystyle{
	\left( \begin{array}{c} F_\infty (s) \\ G_\infty (s) \end{array} \right)
	}
	\\
	& \quad
	\displaystyle{
	+ \,
	\frac 1c U(s)\int_0^s
	U^{-1} (t)
	\left\{  \left(\begin{array}{c} 1 \\ 0 \end{array} \right)
	+  
	\frac{(n-3)G_{c}(t)}{F_{c}(t)^{2}+G_{c}(t)^{2}} \left( \begin{array}{c} G_c (t) \\ -F_c (t) \end{array}
	\right) \right\} dt}
	.
	\end{array}
	\label{Eq.(7)}
\end{equation}
Setting $\epsilon =1/c$, we have the following asymptotic expansions of $F_{c}(s)$ and $G_{c}(s)$.
\begin{thm} For a continuous $H(s)$ on $\mathbb{R}$, it holds that
\[
	F_c (s) = \sum_{ k=0 }^2 \epsilon^k F_\infty^{(k)} (s) + {\cal O} ( \epsilon^{ 3 } ) ,
	\quad
	G_c (s) = \sum_{ k=0 }^2 \epsilon^k G_\infty^{(k)} (s) + {\cal O} ( \epsilon^{ 3 } ) ,
\]
where 
\begin{itemize}
\item[{\rm (i)}]
$ \displaystyle{
\left(
\begin{array}{c}
F_{\infty}^{(0)}(s)  \\
G_{\infty}^{(0)}(s) 
\end{array}
\right)
=
\left(
\begin{array}{c}
F_{\infty}(s)  \\
G_{\infty}(s) 
\end{array}
\right)
} $,
\item[{\rm (ii)}]
$ \displaystyle{
\left(
\begin{array}{c}
F_{\infty}^{(1)}(s)  \\
G_{\infty}^{(1)}(s) 
\end{array}
\right)
=
U(s) \int_{0}^{s}\left(
\begin{array}{r}
(n-2) \cos \eta(t)  \\
-\sin \eta(t)
\end{array}
\right) dt
} $,
\item[{\rm (iii)}]
$ \displaystyle{
\left(
\begin{array}{c}
F_{\infty}^{(2)}(s)  \\
G_{\infty}^{(2)}(s) 
\end{array}
\right)
=
(n-2)(n-3)U(s) \int_{0}^{s} \left(\int_{0}^{t} \cos \eta(u) du
\left(
\begin{array}{r}
 \sin \eta(t)  \\
-\cos \eta(t) 
\end{array}
\right) \right)dt
} $.
\end{itemize}
\label{thm.3.3}
\end{thm}
\begin{proof}
(i) follows from Theorem \ref{thm.3.2}.
To prove (ii),
we compute formally,
by \pref{Eq.(7)},
 \[
	\begin{array}{rl}
	\displaystyle{
	\left( \begin{array}{c}
	F_\infty^{(1)} (s) \\ G_\infty^{(1)} (s)
	\end{array} \right)
	}
	= & 
	\displaystyle{
	\lim_{ \epsilon \to 0 }
	\frac 1 \epsilon
	\left\{
	\left( \begin{array}{c}
	F_c (s) \\ G_c (s)
	\end{array} \right)
	-
	\left( \begin{array}{c}
	F_\infty^{(0)} (s) \\ G_\infty^{(0)} (s)
	\end{array} \right)
	\right\}
	}
	\\
	= & 
	\displaystyle{
	U(s) \int_0^s U^{-1} (t)
	\left\{
	\left( \begin{array}{c} 1 \\ 0 \end{array} \right)
	+
	\frac { ( n - 3 ) G_\infty (t) }
	{ F_\infty (t)^2 + G_\infty (t)^2 }
	\left( \begin{array}{r}
	G_\infty (t) \\ - F_\infty (t)
	\end{array} \right)
	\right\}
	dt
	}
	\\
	= & 
	\displaystyle{
	U(s) \int_0^s
	\left( \begin{array}{r}
	( n - 2 ) \cos \eta (t) \\ - \sin \eta (t)
	\end{array} \right)
	dt
	}
	.
	\end{array}
\]
 It holds for this that
\[
	\begin{array}{l}
	\displaystyle{
	\left( \begin{array}{c}
	F_c (s) \\ G_c (s)
	\end{array} \right)
	-
	\left( \begin{array}{c}
	F_\infty^{(0)} (s) \\ G_\infty^{(0)} (s)
	\end{array} \right)
	-
	\epsilon
	\left( \begin{array}{c}
	F_\infty^{(1)} (s) \\ G_\infty^{(1)} (s)
	\end{array} \right)
	}
	\\
	\quad
	=
	\displaystyle{
	\epsilon ( n - 3 ) U(s) \int_0^s U^{-1} (t)
	\left\{
	\frac { G_c (t) } { F_c (t)^2 + G_c (t)^2 }
	\left( \begin{array}{r}
	G_c (t) \\ - F_c (t)
	\end{array} \right)
	\right.
	}
	\\
	\qquad \qquad
	\qquad \qquad
	\qquad \qquad
	\qquad \qquad
	\displaystyle{
	\left.
	- \,
	\frac { G_\infty (t) } { F_\infty (t)^2 + G_\infty (t)^2 }
	\left( \begin{array}{r}
	G_\infty (t) \\ - F_\infty (t)
	\end{array} \right)
	\right\}
	dt
	}
	.
	\end{array}
\]
For any compact interval $I$,
there exists $C_{I}$ such that
\begin{eqnarray*}
&& \sup_{t \in I} \left|\frac{G_{c}(t)}{F_{c}(t)^2 + G_{c}(t)^2} \left(
\begin{array}{r}
G_{c}(t)  \\
-F_{c}(t) 
\end{array}
\right) - \frac{G_{\infty}(t)}{F_{\infty}(t)^2+G_{\infty}(t)^2}\left(
\begin{array}{r}
G_{\infty} (t)  \\
- F_{\infty}(t) 
\end{array}
\right) \right| \\
&& \leq C_{I}\left(\sup_{t \in I}|G_{c}(t) 
- G_{\infty}(t)| + \sup_{t \in I} |F_{c}(t) - F_{\infty}(t)|  \right) \leq C_{I} \epsilon.
\end{eqnarray*}
 This estimate above implies that
 $$
\left(
\begin{array}{c}
F_{c}(s)  \\
G_{c}(s) 
\end{array}
\right)  - \left(
\begin{array}{c}
F_{\infty}^{(0)}(s)  \\
G_{\infty}^{(0)}(s) 
\end{array}
\right) - \epsilon U(s)\int_{0}^{s}\left(
\begin{array}{c}
(n-2)\cos \eta(t)  \\
-\sin \eta(t) 
\end{array}
\right)dt = {\cal O} ( \epsilon^{ 2 } )
$$
and that the convergence
$$
\frac{1}{\epsilon}\left\{ \left(
\begin{array}{c}
F_{c}(s)  \\
G_{c}(s) 
\end{array}
\right)  - \left(
\begin{array}{c}
F_{\infty}^{(0)}(s)  \\
G_{\infty}^{(0)}(s) 
\end{array}
\right) \right\} \to   U(s)\int_{0}^{s}\left(
\begin{array}{c}
(n-2)\cos \eta(t)  \\
-\sin \eta(t) 
\end{array}
\right)dt \ \  \mbox{as} \ \ \epsilon \rightarrow 0
$$
is compactly uniform, which proves (ii).
\par
For the proof of (iii),
 the above consideration yields 
\[
	\begin{array}{l}
	\displaystyle{
	\left( \begin{array}{c}
	F_\infty^{(2)} (s) \\ G_\infty^{(2)} (s)
	\end{array} \right)
	=
	\lim_{ \epsilon \to 0 }
	\frac 1 { \epsilon^2 }
	\left\{
	\left( \begin{array}{c}
	F_c (s) \\ G_c (s)
	\end{array} \right)
	-
	\left( \begin{array}{c}
	F_\infty^{(0)} (s) \\ G_\infty^{(0)} (s)
	\end{array} \right)
	-
	\epsilon
	\left( \begin{array}{c}
	F_\infty^{(1)} (s) \\ G_\infty^{(1)} (s)
	\end{array} \right)
	\right\}
	}
	\\
	\quad
	=
	\displaystyle{
	\lim_{ \epsilon \to 0 }
	\frac { n - 3 } \epsilon U(s) \int_0^s U^{-1} (t)
	\left\{
	\frac { G_c (t) } { F_c (t)^2 + G_c (t)^2 }
	\left( \begin{array}{r}
	G_c (t) \\ - F_c (t)
	\end{array} \right)
	\right.
	}
	\\
	\qquad \qquad
	\qquad \qquad
	\qquad \qquad
	\qquad \qquad
	\displaystyle{
	\left.
	- \,
	\frac { G_\infty (t) } { F_\infty (t)^2 + G_\infty (t)^2 }
	\left( \begin{array}{r}
	G_\infty (t) \\ - F_\infty (t)
	\end{array} \right)
	\right\}
	dt
	}
	.
	\end{array}
\]
For any compact interval $I$, there exists $C_{I}$ such that
\[
	\begin{array}{l}
	\displaystyle{
	\sup_{ t \in I }
   \frac 1 \epsilon 
    \left|
	\frac { G_c (t) } { F_c (t)^2 + G_c (t)^2 }
	\left( \begin{array}{r}
	G_c (t) \\ - F_c (t)
	\end{array} \right)
	-
	\frac { G_\infty (t) } { F_\infty (t)^2 + G_\infty (t)^2 }
	\left( \begin{array}{r}
	G_\infty (t) \\ - F_\infty (t)
	\end{array} \right)
	\right|
	}
	\\
	\quad
	\leqq
	\displaystyle{
	 \frac { C_I } \epsilon 
	\left(	\sup_{ t \in I } \left| G_c (t) - G_\infty (t) \right| + 
	\sup_{ t \in I } \left| F_c (t) - F_\infty (t) \right|
	\right)
	\leqq
	C_I
	}
	.
	\end{array}
\]
Therefore we can exchange the order of the limit as $ \epsilon \to 0 $ and the intergation by the dominated convergence theorem,
and obtain
\[
	\begin{array}{l}
	\displaystyle{
	\left( \begin{array}{c}
	F_\infty^{(2)} (s) \\ G_\infty^{(2)} (s)
	\end{array} \right)
	}
	\\
	\quad
	=
	\displaystyle{
	( n - 3 ) U(s) \int_0^s U^{-1} (t)
	\lim_{ \epsilon \to 0 }
	\frac 1 \epsilon
	\left\{
	\frac { G_c (t) } { F_c (t)^2 + G_c (t)^2 }
	\left( \begin{array}{r}
	G_c (t) \\ - F_c (t)
	\end{array} \right)
	\right.
	}
	\\
	\qquad \qquad
	\qquad \qquad
	\qquad \qquad
	\qquad \qquad
	\displaystyle{
	\left.
	- \,
	\frac { G_\infty (t) } { F_\infty (t)^2 + G_\infty (t)^2 }
	\left( \begin{array}{r}
	G_\infty (t) \\ - F_\infty (t)
	\end{array} \right)
	\right\}
	dt
	}
	\\
	\quad
	=
	\displaystyle{
	( n - 3 ) U(s) \int_0^s U^{-1} (t)
	\left.
	\frac \partial { \partial \epsilon }
	\left\{
	\frac { G_c (t) } { F_c (t)^2 + G_c (t)^2 }
	\left( \begin{array}{r}
	G_c (t) \\ - F_c (t)
	\end{array} \right)
	\right\}
	\right|_{ \epsilon = 0 }
	dt
	}
	.
	\end{array}
\]
Since
\[
	\begin{array}{l}
	\displaystyle{
	\left.
	\frac \partial { \partial \epsilon }
	\left\{
	\frac { G_c (t) } { F_c (t)^2 + G_c (t)^2 }
	\left( \begin{array}{r}
	G_c (t) \\ - F_c (t)
	\end{array} \right)
	\right\}
	\right|_{ \epsilon = 0 }
	}
	\\
	\quad
	=
	\displaystyle{
	\left\{
	\frac { G_\infty^{(1)} (t) }
	{ F_\infty^{(0)} (t)^2 + G_\infty^{(0)} (t)^2 }
	-
	\frac { 2 G_\infty^{(0)} (t)
	\left(
	F_\infty^{(0)} (t) F_\infty^{(1)} (t)
	+
	G_\infty^{(0)} (t) G_\infty^{(1)} (t)
	\right) }
	{ \left( F_\infty^{(0)} (t)^2 + G_\infty^{(0)} (t)^2 \right)^2 }
	\right\}
	\left( \begin{array}{r}
	G_\infty^{(0)} (t) \\ - F_\infty^{(0)} (t)
	\end{array} \right)
	}
	\\
	\quad \qquad
	\displaystyle{
	+ \,
	\frac { G_\infty^{(0)} (t) }
	{ F_\infty^{(0)} (t)^2 + G_\infty^{(0)} (t)^2 }
	\left( \begin{array}{r}
	G_\infty^{(1)} (t) \\ - F_\infty^{(1)} (t)
	\end{array} \right)
	}
	\\
	\quad
	=
	\displaystyle{
	\left( \begin{array}{rr}
	- 2 G_\infty^{(0)} (t)^2 F_\infty^{(0)} (t) , &
	 G_\infty^{(0)} (t) \left( 1 - 2G_\infty^{(0)} (t)^2 \right) \\
	G_\infty^{(0)} (t) \left( 2 F_\infty^{(0)} (t)^2 - 1 \right), &
	F_\infty^{(0)} (t) \left( 2 G_\infty^{(0)} (t)^2 - 1 \right)
	\end{array} \right)
	\left( \begin{array}{c}
	F_\infty^{(1)} (t) \\ G_\infty^{(1)} (t)
	\end{array} \right)
	}
	\\
	\quad
	=
	\displaystyle{
	\left( \begin{array}{rr}
	\cos \eta (t) \sin 2 \eta (t) &
	\sin \eta (t) \sin 2 \eta (t) \\
	- \cos \eta (t) \cos 2 \eta (t) &
	- \sin \eta (t) \cos 2 \eta (t)
	\end{array} \right)
	\left( \begin{array}{c}
	F_\infty^{(1)} (t) \\ G_\infty^{(1)} (t)
	\end{array} \right)
	}
	\\
	\quad
	=
	\displaystyle{
	\left( \begin{array}{r}
	\sin 2 \eta (t) \\ - \cos 2 \eta (t)
	\end{array} \right)
	( \cos \eta (t) \ \ \sin \eta (t) )
	U(t) \int_0^t
	\left( \begin{array}{r}
	( n - 2 ) \cos \eta (u) \\ - \sin \eta (u)
	\end{array} \right)
	du
	}
	\\
	\quad
	=
	\displaystyle{
	\left( \begin{array}{r}
	\sin 2 \eta (t) \\ - \cos 2 \eta (t)
	\end{array} \right)
	( 1 \ \ 0 )
	\int_0^t
	\left( \begin{array}{r}
	( n - 2 ) \cos \eta (u) \\ - \sin \eta (u)
	\end{array} \right)
	du
	}
	\\
	\quad
	=
	\displaystyle{
	( n - 2 )
	\left( \begin{array}{r}
	\sin 2 \eta (t) \\ - \cos 2 \eta (t)
	\end{array} \right)
	\int_0^t
	\cos \eta (u)
	\, du
	}
	,
	\end{array}
\]
we have
\[
	\begin{array}{rl}
	\displaystyle{
	\left( \begin{array}{c}
	F_\infty^{(2)} (s) \\ G_\infty^{(2)} (s)
	\end{array} \right)
	}
	= & \!\!\!
	\displaystyle{
	( n - 2 ) ( n - 3 ) U(s) \int_0^s
	U^{-1} (t)
	\left( \begin{array}{r}
	\sin 2 \eta (t) \\ - \cos 2 \eta (t)
	\end{array} \right)
	\int_0^t
	\cos \eta (u) \, du
	dt
	}
	\\
	= & \!\!\!
	\displaystyle{
	( n - 2 ) ( n - 3 ) U(s) \int_0^s
	\left( \begin{array}{r}
	\sin \eta (t) \\ - \cos \eta (t)
	\end{array} \right)
	\int_0^t
	\cos \eta (u) \, du
	dt
	}
	,
	\end{array}
\]
proving Theorem \ref{thm.3.3}.
\qed
\end{proof}

\begin{rem}
Similary we can show, for any $ K \in \mathbb{N} $,
\[
	F_c (s) = \sum_{ k=0 }^K \epsilon^k F_\infty^{(k)} (s) + {\cal O} ( \epsilon^{ K+1 } ) ,
	\quad
	G_c (s) = \sum_{ k=0 }^K \epsilon^k G_\infty^{(k)} (s) + {\cal O} ( \epsilon^{ K+1 } ) ,
\]
where, $ k \geqq 2 $, 
\[
	\displaystyle{
	\left( \begin{array}{c}
	F_\infty^{(k)} (s) \\ G_\infty^{(k)} (s)
	\end{array} \right)
	}
	\\
	\quad
	=
	\displaystyle{
	\frac { n - 3 } { ( k - 1 ) ! }
	U(s) \int_0^s U^{-1} (t)
	\left.
	\frac { \partial^{ k-1 } } { \partial \epsilon^{ k-1 } }
	\left\{
	\frac { G_c (t) } { F_c (t)^2 + G_c (t)^2 }
	\left( \begin{array}{r}
	G_c (t) \\ - F_c (t)
	\end{array} \right)
	\right\}
	\right|_{ \epsilon = 0 }
	dt.
}
\]
\end{rem}

\par
Now we shall study periodicity of the family  $\{\Gamma_{c}\}$.  If $H(s)$ is non-zero constant, then $\Gamma_{c}$ is 
periodic  for any $c>0$  \cite{del,hsyu} and $\Gamma_{\infty}$ is a round circle with the curvature 
$(n-1)|H(s)|$. The period of $\Gamma_{c}$ depends on the initial condition of the genrating curve when $n>3$. If $H(s)$ is not constant and $n>3$, then the situation is different from the constant mean curvature case and also 2-dimensional case \cite{ken2}.

\begin{lem} Suppose that for any $c>0$, $\Gamma_{c}$ is periodic with period $L(c)$. If $H(s)$ and $L(c)$ are differentiable for $s$ and $c$ respectively,  then one of them is constant.
\label{lem.3.1}
\end{lem}
\begin{proof} Since $x_{c}$ and $y_{c}$ are periodic for every $c>0$, we have $H(s) =H(s+L(c)), \ \mbox{for} \  s \in \mathbb{R} \ \mbox{and} \ c>0 $. Differentiating both side above with respect to $c$, we have $0 = H'(s+L(c))L'(c)$, which proves Lemma \ref{lem.3.1}.
\qed
\end{proof}
\par
We have
\begin{thm}  Let $n>3$ and $H(s)$ be a non-constant differentiable function on $\mathbb{R}$. Then,
 there does not exist a family $\{\Gamma_{c}\}$ such that 
 \begin{itemize}
 \item[{\rm (i)}] $\Gamma_{c}$ is periodic with period $L(c)$,
 \item[{\rm (ii)}] $L(c)$ is differentiable for $c>0$,
 \item[{\rm (iii)}]	$\Gamma_{\infty}$ is a simple closed curve.
\end{itemize}
\label{thm.3.4}
\end{thm}
\begin{proof} Suppose that there exists a family $\{\Gamma_{c}\}$  satisfying these three conditions in Theorem \ref{thm.3.4}.
By Lemma \ref{lem.3.1},
$L(c)$ is constant, say $L(c)=L>0$. Since $\Gamma_{c}$ is periodic with period $L$ for any $c>0$, $F_{c}(s)$ and $G_{c}(s)$ are  periodic with period $L$ for any $c>0$.  Then, 
$F_{\infty}^{(k)}(s)$ and $G_{\infty}^{(k)}(s)$ are also  periodic with period $L$.
By Theorem \ref{thm.3.3}, 
we have
\begin{eqnarray*}
 &&  \int_{0}^{L} \cos \eta(u) du =0, \ \int_{0}^{L} \sin \eta(u) du =0, \\
 && (n-2)(n-3) \int_{0}^{L} \left(\sin \eta(u) \int_{0}^{u}\cos \eta(t)dt \right)du =0.
\end{eqnarray*}
The first formulas above with \pref{Eq.(6)} implies that $\Gamma_{\infty}$ is a closed smooth curve and the second 
one means that the signed area $A(\Gamma_{\infty})$ of the bounded domain surrounded by
 $\Gamma_{\infty}$ is zero if $n>3$. In fact,  $A(\Gamma_{\infty})$  is given by
 $A(\Gamma_{\infty}) = \left| \int_{0}^{L}x_{\infty}(u)y_{\infty}'(u)du \right|$  and it is 
 positive if 
$\Gamma_{\infty}$ is a simple closed curve, which  gives us contradiction,  proving Theorem \ref{thm.3.4}.
\qed
\end{proof}
\par
 
 \section{Global existence of generalized rotational hypersurfaces of $O(l+1)\times O(m+1)$-type} 
In this section,
we shall study generalized rotational hypersurfaces in $\mathbb{R}^{n}$ of $O(l+1)\times O(m+1)$-type with 
prescribed mean curvature.
The main task is the analysis of the behavior of the generating curve in the neighborhood of singular points.
This is done by getting an integral equation instead of the system of the  differential equations. 
\par
Let $(x(s),y(s))$,
$ s \in \mathbb{R} $,
be a plane curve satisfying $ x(s) > 0 $,
$ y(s) > 0$  and $s$ be the  arc length.
For natural numbers $l$ and $m$ with $(l+1)+(m+1) = n$,
we decompose $\mathbb{R}^{n}$ as $ (x_{1}, \cdots , x_{l+1}, y_{1}, \cdots , y_{m+1}) \in \mathbb{R}^{l+1} \times \mathbb{R}^{m+1}$.
A generalized rotational hypersurface $M$ of $O(l+1)\times O(m+1)$-type is defined by
\begin{equation}
	M =
	\left\{
	(x(s) S^{l} , y(s)S^{m}) \in \mathbb{R}^{n} \ |\  s \in \mathbb{R}
	\right\},
	\label{M}
\end{equation}
where,
for $k=l$,
$m$,
$S^{k}$ denotes the $k$-dimensional unit sphere with center origin in $\mathbb{R}^{k+1}$.
Note that this is $ O( m+1 ) \times O( l+1 ) $-type in the term of \cite{hsjdg}.
The mean curvature $H$ of $M$ is the function of one variable $s$, say $H=H(s)$, and we have 
\begin{equation}
	\left\{
	\begin{array}{l}
	\displaystyle{
	l \frac { y^\prime (s) } { x(s) }
	- m \frac { x^\prime (s) } { y(s) }
	- ( x^{\prime\prime}(s) y^\prime(s) - x^\prime(s)y^{\prime\prime} (s) )
	+ ( n - 1 ) H(s)
	= 0
	} ,
	\\
	x^\prime (s)^{2} + y^\prime (s)^{2} =1 ,
	\ x(s) > 0 , \ y(s) > 0, \
	s \in \mathbb{R}.
	\end{array}
	\right. 
	\label{lm-system}
\end{equation}

Conversely, given a continuous function $H=H(s)$,
$ s \in \mathbb{R}$,  we have the system of ordinary differential equations \pref{lm-system}.
Let $c>0$ and $d>0$ be any positive numbers  and any real numbers $c^\prime $,
$ d^\prime$ with ${c^\prime}^2 + {d^\prime}^{2} = 1$,
there exists a local solution curve $(x(s),y(s))$,
$ s \in I$, where $I$ denotes a subinterval of $\mathbb{R}$,
 of \pref{lm-system} with the initial conditions $x(s_{0})=c$,
$ y(s_{0})=d$,
$ x^\prime(s_{0})=c^\prime$,
$ y^\prime(s_{0})=d^\prime$.

\par
We shall extend this curve to the whole line $\mathbb{R}$.
The problem to be studied is the behavior of the solution curve when it passes through $x$-axis,
$y$-axis,
or the origin $(0,0)$.
When the curve passes through $y$-axis,
by changing $x$ and $y$,
we can analyze the behavior of the solution curve in the same way as the curve passes through $x$-axis.
Therefore,
it is sufficient to study the two cases that (a):
the curve $(x(s),y(s))$ passes through $x$-axis at a finite $s$,
and (b):
the curve $(x(s),y(s))$ passes through the origin $(0,0)$. 
We shall study cases (a) and (b) in \S\S~4.1 and 4.2 respectively.
\subsection{The case $(a)$}
Let us multiply $ y^{m}y^\prime $ to the first equation of \pref{lm-system}.
By  the second equation of \pref{lm-system} and the equation $ x^\prime x^{\prime\prime} + y^\prime y^{\prime\prime} = 0 $ 
which is obtained by differentiation of the second equation of \pref{lm-system},
we have
\begin{equation}
	( y^m x^\prime )^\prime
	=
	( n - 1 ) H(s) y^m y^\prime + l \frac { y^m } {x} { y^\prime}^{2} .
	\label{eqy}
\end{equation}
For simplicity we assume $ s_0 = 0 $.
The integration of \pref{eqy} leads to
\begin{equation}
	\begin{array}{rl}
	y^m (s) x^\prime (s)
	= & \!\!\!
	\displaystyle{
	( n - 1 ) \int_0^s H(t) y^m (t) y^\prime(t) \, dt
	}
	\\
	& \qquad
	\displaystyle{
	+ \,
	l \int_0^s \frac { y^m (t) } { x(t) } y^\prime (t)^{2} dt
	+ y^m (0) x'(0)
	}
	.
	\end{array}
	\label{eqy1}
\end{equation}
We show that
\begin{prop}
Suppose that $ \lim_{s \rightarrow b} y(s) = 0 $,
and $ \lim_{s \rightarrow b} x(s) > 0 $ for some $ b \in I$.
Then,
there exists the limit of $ x^\prime(s) $ as $ s \to b $,
and $\lim_{s \rightarrow b} x^\prime (s) = 0 $.
\label{prop.4.1}
\end{prop}
\par
To prove Proposition \ref{prop.4.1},
we need the following Lemma.
\begin{lem}
$ y^\prime (s) $ does not vanish on a neighborhood of $ s=b $.
\label{lem.4.1}
\end{lem}
\begin{proof}
By contraries,
suppose that there exists a sequence $ \{ s_j \} $ such that
$ s_j \to b $ as $ j \to \infty $,
and $ y^\prime ( s_j ) = 0 $.  
The  second equation of \pref{lm-system} and its differentiation imply
$ x^\prime ( s_j ) = \pm 1$,
$ x^{\prime\prime} ( s_j ) = 0$.
By inserting these to the  first equation of \pref{lm-system},
we have 
\[
	( n - 1 ) H( s_j )
	\mp \frac m { y( s_j ) } \pm y^{ \prime \prime } ( s_j ) = 0.
\]
This with $ y( s_j ) \to + 0 $ as $ j \to \infty $ yields
\[
	y^{ \prime \prime } ( s_j )
	= \mp ( n - 1 ) H( s_j ) + \frac m { y( s_j ) }
	\to \infty \mbox{ as } j \to \infty .
\]
Hence,
for large $ j $,
$ y( s_j ) $ is the minimum.
Consequently,
it does not hold that $ y( s ) \to +0 $ as $ s \to b $,
giving contradiction. We proved Lemma \ref{lem.4.1}.
\qed
\end{proof}
\par
We proceed now the proof of Proposition \ref{prop.4.1}.
\begin{proof}[Proof of Proposition \ref{prop.4.1}]
Since $ y(s) \to 0 $ as $ s \to b $,
\pref{eqy1} implies
$$
	( n - 1 ) \int_0^b H(t) y^m(t) y^\prime (t) \, dt
	+ l \int_0^b \frac { y^m (t) } { x(t) } y^\prime(t)^2 dt
	+ y^m (0) x^\prime (0)
	= 0 .
$$
From \pref{eqy1},
Lemma \ref{lem.4.1},
and l'H\^{o}pital's rule,
it follows that
\begin{eqnarray*}
	\lim_{s \rightarrow b} x^\prime (s)
	& \!\!\!\! = & \!\!\!
	\lim_{s \rightarrow b}
	\frac{ \displaystyle{
	( n - 1 ) \int_0^s
	H(t) y^m (t) y^\prime (t) \, dt
	+ l \int_0^s \frac { y^m (t) } { x(t) } y^\prime(t)^2 dt
	+ y^m (0) x^\prime (0)
	} }
	{ y^m (s) }
	\\
 	& \!\!\! = & \!\!\!
 	\lim_{s \rightarrow b}
 	\frac{ \displaystyle{
 	( n - 1 ) H(s) y^m (s) y^\prime (s)
 	+ l \frac { y^m (s) }{ x(s) } y^\prime (s)^2
 	} }
 	{ m y^{m-1} (s) y^\prime (s)}
 	\\
	& \!\!\! = & \!\!\!
	\frac 1m \lim_{ s \to b }
	\left\{ ( n - 1 ) H(s) y(s) + l \frac { y(s) } { x(s) } y^\prime (s)\right\}
	= 0,
\end{eqnarray*}
proving Proposition \ref{prop.4.1}.
\qed
\end{proof}
\par
In order to prove the existence of solutions of the system \pref{lm-system} under the 
assumption of Proposition \ref{prop.4.1},
we transform \pref{lm-system} by a change of variable.
Let $s=s(y)$ be the inverse function of $y=y(s)$ on a neighborhood of $s=0$.
Let us pose $q(s) = \frac{x'(s)}{y'(s)}$.
We then obtain 
\begin{eqnarray*}
	x y \frac {dq} {dy}
	& \!\!\! = & \!\!\!
	( n - 1 ) \tilde H (y) ( 1 + q^2)^{3/2} xy
	- mx q ( 1 + q^2 ) 
	+ ly ( 1 + q^2 ),
	\\
	x
	& \!\!\! = & \!\!\!
	\int_0^y q ( \xi ) \,d \xi + x(0),
\end{eqnarray*}
where $ \tilde H(y) $ is defined by the same way as in \pref{DiffEq(8)}.
\par
We,
then,
have the singular initial value problem
\begin{equation}
	\quad \,\,
	\left\{
	\begin{array}{l}
	\displaystyle{
	y \frac {dq} {dy}
	=
	- mq - mq^{3} + ( n - 1 ) \tilde H (y) y ( 1 + q^2 )^{\frac 32}
	+ l \frac { y ( 1 + q^2 ) }
	{ \displaystyle{ \int_0^y q(\xi) \, d \xi + x(0)} }
	},
	\\
	q(0) = 0.
	\end{array}
	\right.
	\label{DiffEq2}
\end{equation}
\begin{rem}
Although it is supposed to be $ l \geq 1 $ in this section,
if $ l=0 $ and $ m = n - 2 $,
then \pref{DiffEq2} is reduced to \pref{DiffEq(8)} of $O(n-1)$-type treated in the previous section.
When $l\geq 1$, the last term in the right hand of the first equation of \pref{DiffEq2} is new one.
\end{rem}
\par
We  study \pref{DiffEq2} as follows.
It follows from \pref{DiffEq2} that
$$
	\frac d {dy} (y^m q) 
	=
	- m y^{m-1} q^3
	+ ( n - 1 ) \tilde H (y) y^m ( 1 + q^2 )^{3/2}
	+ l \frac { y^m ( 1 + q^2 ) }
	{ \displaystyle{
	\int_0^y q(\xi) \, d\xi + x(0)
	} }
	.
$$
Let us integrate both sides above.
By virtue of $ \left. y^m q \right|_{y=0} = 0 $,
we have an integral equation 
\begin{equation}
	q(y) = \Psi(q)(y),
	\label{Psi}
\end{equation}
where we set
$$
	\Psi(q)(y)
	 = 
	y^{-m} \int_0^y
	\left\{
	- m q(\eta)^3
	+ ( n - 1) \tilde H (\eta) \eta (1 + q(\eta)^2 )^{3/2}
	\vphantom{ \frac { )^2 } { \displaystyle{ \int_0^\eta } } }
		+ \,
	l \frac { \eta ( 1 + q(\eta)^2 )}
	{ \displaystyle{
	\int_0^\eta q(\xi) \, d\xi + x(0)
	} }
	\right\}
	\eta^{m-1} d\eta .
$$
We note that when $ l=0 $,
$ m = n - 2$,
\pref{Psi} is reduced to \pref{IntEq}.
We use the same notations $X_{Y}$ and $X_{Y,M}$,
which are defined in the previous section,
in particular,
Proposition \ref{prop.2.2}.
\begin{prop}
Suppose that $ x(0) > 0 $ and $ y(0) = 0 $.
\begin{itemize}
\item[{\rm (i)}]
If $ \tilde H $ is bounded,
then there exist constants $M$ and $Y$ such that the integral equation {\rm \pref{Psi}} has
a unique solution $q$ on $X_{M}$.
\item[{\rm (ii)}]
If $ \tilde H $ is bounded and continuous,
then the solution $q$ obtained in {\rm (i)} is a unique solution of the 
initial value problem {\rm \pref{DiffEq2}}.
\end{itemize}
\label{prop.4.2}
\end{prop}
\begin{proof}
For any $ q \in X_{Y,M} $,
we know $ \Psi(q) \in C(0,Y] $.
To prove (i),
we show the boundedness of $\Psi(q)$ 
as follows:
By noting that
$$
	\left| \int_0^\eta q(\xi) \, d\xi \right|
	\leqq
	M \int_{0}^{\eta}\xi \, d \xi
	= \frac M2 \eta^2 ,
$$
we choose $M$ and $Y$ such that
$$
	\left|
	\int_0^\eta q(\xi) \, d\xi + x(0)
	\right|
	\geqq
	x(0) - \frac M2 \eta^2
	\geqq
	x(0) - \frac M2 Y^2 > 0 .
$$
Then,
\begin{eqnarray*}
	\left|
	\frac 1 y y^{-m}
	\int_0^y
	\frac { l \eta^m ( 1 + q(\eta)^2 ) }
	{ \displaystyle{
	\int_0^\eta q(\xi)\, d\xi + x(0)
	} }
	\, d \eta
	\right|
	& \!\!\! \leqq & \!\!\!
	\frac { \displaystyle{
	l \int_0^y \left( \eta^m + M^2 \eta^{m+2} \right) d\eta
	} }
	{ y^{m+1} \left( x(0) - \frac M2 Y^2 \right) }
	\\
	& \!\!\! = & \!\!\!
	\frac { \displaystyle{
	l \left(
	\frac{ y^{m+1} } { m+1 }
	+ \frac{ M^2 } { m + 3 } y^{m+3}
	\right)
	} }
	{ \displaystyle{
	y^{m+1} \left( x(0) - \frac M2 Y^2 \right)
	} }
	\\
	& \!\!\! \leqq & \!\!\!
	\frac{ \displaystyle{
	l \left( \frac 1 { m + 1 } + \frac { M^2 } { m + 3 } Y^2 \right)
	} }
	{ \displaystyle{ x(0) - \frac M2 Y^{2} } } .
\end{eqnarray*}
Since the estimates of other parts in the right hand of \pref{Psi} are done by the same way as the proof of Proposition \ref{prop.2.2} in \S~2,
there is a constant $C$ such that
$$
	\left|
	\frac { \Psi(q)(y) } {y}
	\right|
	\leqq
	C \left( M^3 Y^2 + 1 + M^3 Y^3 \right)
	+ \frac{ \displaystyle{
	l \left( \frac 1 { m + 1 } + \frac{ M^2 Y^2 } { m + 3 } \right)
	} }
	{ \displaystyle{ x(0) - \frac M2 Y^2 } } ,
$$
where $C$ may depend on $ l $,
$ m $,
$n$ and $\sup_{y}| \tilde H(y) |$,
but independent of $M$ and $Y$.
We can choose a large $M$ and a small $Y$ such that
$$
	C \left( M^3 Y^2 + 1 + M^3 Y^3 \right)
	+ \frac{ \displaystyle{
	l \left( \frac 1 { m + 1 } + \frac { M^2 Y^2 } { m + 3 } \right)
	} }
	{ \displaystyle{ x(0) - \frac M2 Y^2 } }
	\leqq M ,
$$
so that $\Psi(q) \in X_{Y,M}$.
\par
Next,
we show that $ \Psi : X_M \longrightarrow X_M $ is a contraction mapping.
Let us decompose $ \Psi(q) $ as $\Psi(q) =\Psi_1 (q) + \Psi_2 (q)$,
where
\begin{eqnarray*}
	\Psi_1 (q)(y)
	& \!\!\! = & \!\!\!
	y^{-m} \int_0^y
	\left\{
	- m \eta^{m-1} q(\eta)^3
	+ ( n - 1) \tilde H (\eta) \eta^m ( 1 + q(\eta) )^2)^{3/2}
	\right\} d \eta ,
	\\
	\Psi_2 (q)(y)
	& \!\!\! = & \!\!\!
	l y^{-m} \int_0^y
	\frac{ \eta^m ( 1 + q(\eta)^2 ) }
	{ \displaystyle{
	\int_0^\eta q(\xi) \, d\xi + x(0)
	} }
	\, d\eta .
\end{eqnarray*}
\par
For $\Psi_1 (q)$,
by using the previous result,
there exists a constant $ C_1 (\in (0,1)) $ such that
$$
	\left|
	\frac{ \Psi_1 ( q_1 )(y) - \Psi_1 ( q_2 )(y)}{y} \right|
	<
	C_1 \Vert q_1 - q_2 \Vert_{X}.
$$
\par
For $\Psi_2 (q)$, we compute
\begin{eqnarray*}
	&&
	\Psi_2 ( q_1 )(y) - \Psi_2 ( q_2 )(y)
	\\
	&&
	=
	l y^{-m}
	\int_0^y
	\left\{
	\frac{ \eta^m ( 1 + q_1 (\eta)^2 )}
	{ \displaystyle{
	\int_0^\eta q_1 (\xi) \, d\xi + x(0)
	} }
	-
	\frac{ \eta^m ( 1 + q_2 (\eta)^2 )}
	{ \displaystyle{
	\int_0^\eta q_2 (\xi) \, d\xi + x(0)
	} }
	\right\}
	d\eta
	\\
	&&
	=
	l y^{-m}
	\left\{
	\int_0^y
	\frac{
	\displaystyle{
	\eta^m ( 1 + q_1 (\eta)^2 )
	\left( \int_0^\eta q_2 (\xi) \, d\xi + x(0) \right)
	} }
	{ \displaystyle{
	\left( \int_0^\eta q_1 (\xi) \, d\xi + x(0) \right)
	\left( \int_0^\eta q_2 (\xi) \, d\xi + x(0) \right)
	} }
	d\eta
	\right.
	\\
	&&
	\qquad \qquad \qquad \qquad
	\left.
	- \,
	\int_0^y
	\frac{ \displaystyle{
	\eta^m ( 1 + q_2 (\eta)^2 )
	\left( \int_0^\eta q_1 (\xi) \, d\xi + x(0) \right)
	} }
	{ \displaystyle{
	\left( \int_0^\eta q_1 (\xi) \, d\xi + x(0) \right)
	\left( \int_0^\eta q_2 (\xi) \, d\xi + x(0) \right)
	} }
	d \eta
	\right\}
	\\
	&&
	=
	l y^{-m}
	\int_0^y
	\frac{ \displaystyle{
	\eta^m \int_0^\eta ( q_2 (\xi) - q_1 (\xi) ) \, d\xi
	} }
	{ \displaystyle{
	\left( \int_0^\eta q_1 (\xi) \, d\xi + x(0) \right)
	\left( \int_0^\eta q_2 (\xi) \, d\xi + x(0) \right)
	} }
	\, d\eta
	\\
	&&
	\qquad
	+ \,
	l y^{-m} \int_0^y
	\frac{ \displaystyle{
	\eta^m \left\{
	q_1 (\eta)^2 \left( \int_0^\eta q_2 (\xi) \, d\xi + x(0) \right)
	-
	q_2 (\eta)^2 \left( \int_0^\eta q_1 (\xi) \, d\xi + x(0) \right)
	\right\}
	} }
	{ \displaystyle{
	\left( \int_0^\eta q_1 (\xi) \, d\xi + x(0) \right)
	\left( \int_0^\eta q_2 (\xi) \, d\xi + x(0) \right)
	} }
	\, d\eta .
\end{eqnarray*}
Let  $\Psi_{21}$ and $\Psi_{22}$ be respectively the first and second terms in the above formula.
We estimate above terms separately as folows.
\begin{eqnarray*}
	\frac{ |\Psi_{21}| }{y}
	& \!\!\! \leqq & \!\!\!
	\frac{ \displaystyle{
	l \left| \int_0^y \eta^m \int_0^\eta ( q_2 (\xi)- q_1 (\xi) ) \, d\xi d\eta \right|
	} }
	{ \displaystyle{ y^{m+1} \left( x(0) - \frac M2 Y^2 \right)^2 } }
	\\
	& \!\!\! \leqq & \!\!\!
	y^{-m-1} C_2 \Vert q_2 - q_1 \Vert_X
	\int_0^y \eta^m \eta^2 \, d \eta
	\\
	& \!\!\! = & \!\!\!
	C_2 y^{-m-1}\Vert q_2 - q_1 \Vert_X
	\frac{ y^{m+3} } { m + 3 }
	\leqq
	\tilde C_2 \Vert q_2 - q_1 \Vert_X y^2 .
\end{eqnarray*}
Since we have
\begin{eqnarray*}
	&&
	\mbox{the numerator of the integrand of } \Psi_{22}
	\\
	&&
	=
	\eta^m \left\{
	q_1^2 \left( \int_0^\eta q_2 (\xi) \, d\xi + x(0) \right)
	-
	q_2^2 \left( \int_0^\eta q_1 (\xi) \, d\xi + x(0) \right)
	\right\}
	\\
	&&
	=
	\eta^m \left\{
	\left( q_1^2 - q_2^2 \right)
	\left( \int_0^\eta q_2 (\xi) \, d\xi + x(0) \right)
	+ q_2^2 \left( \int_0^\eta q_2 ( \xi ) \, d\xi - \int_0^\eta q_1 ( \xi ) \, d\xi \right)
	\right\} ,
\end{eqnarray*}
we see that
\begin{eqnarray*}
	|\Psi_{22}|
	& \!\!\! \leqq & \!\!\!
	\frac C { \displaystyle{
	\left( x(0) - \frac M2 Y^2 \right)^2
	} }
	\\
	& &
	\qquad
	\times
	\left[
	y^{-m} \int_0^y
	\eta^m
	\left\{
	\left| ( q_1 ( \eta ) + q_2 ( \eta ) ) ( q_1 ( \eta ) - q_2 ( \eta ) )
	\left( \int_0^\eta q_2 ( \xi ) \, d\xi + x(0) \right)
	\right|
	\right.
	\right.
	\\
	& & \qquad \qquad \qquad
	\left.
	\left.
	+ \,
	M^2 \eta^2 C \Vert q_2 - q_1 \Vert_X \frac { \eta^2 } 2
	\right\} d \eta
	\right]
	\\
	& \!\!\! \leqq & \!\!\!
	C y^{-m}
	\left\{
	\int_0^y \Vert q_2 - q_1 \Vert_X \eta^{m+2}
	\left( \Vert q_2 \Vert_X \frac{ \eta^2 } 2 + x(0) \right)
	+ \tilde C M^2 \Vert q_2 - q_1 \Vert_X \eta^{m+4}
	\right\} d\eta .
\end{eqnarray*}
Hence,
\begin{eqnarray*}
	\left|\frac{\Psi_{22}} y \right|
	& \!\!\! \leqq & \!\!\!
	y^{-m-1} \left(
	C_3 \frac{ y^{m+5} }{ m + 5 }
	+ C_4 \frac{ y^{m+3} }{ m + 3 }
	+ C_5 \frac{ y^{m+5} }{ m + 5 }
	\right) \Vert q_2 - q_1 \Vert_X
	\\
	& \!\!\! = & \!\!\!
	\left(\tilde C_3 y^4 + \tilde C_4 y^2 \right) \Vert q_2 - q_1 \Vert_X.
\end{eqnarray*}
Consequently,
we have 
\begin{eqnarray*}
	\left| \frac{ \Psi_2 ( q_1 )(y) - \Psi_2 ( q_2 )(y)} y \right|
	& \!\!\! \leqq & \!\!\!
	\hat C_2 \Vert q_2 - q_1 \Vert_X Y^2
	+ \Vert q_2 - q_1 \Vert_X \left( \tilde C_3 Y^4 + \tilde C_4 Y^2 \right)
	\\
	& \!\!\! = & \!\!\!
	\Vert q_2 - q_1 \Vert_X \left( \hat C_3 Y^4 + \hat C_4Y^2 \right).
\end{eqnarray*}
We choose $Y$ such that $ C_1 + \hat C_3 Y^4 + \hat C_4 Y^2 < 1 $.
\par
Summarizing these computations,
for the case of $ x(0) > 0 $ and $ y(0) = 0 $,
we proved that $ \Psi : X_{Y,M} \longrightarrow  X_{Y,M}$ is a contraction mapping,
hence Banach's fixed point theorem implies (i) of Proposition \ref{prop.4.2}.
Since the proof of (ii) is accomplished by 
the similar way as the proof of (ii) in Proposition \ref{prop.2.2},
we finished the proof of Proposition \ref{prop.4.2}.
\qed
\end{proof}
\subsection{The case $(b)$}
In this section,
we study  case (b).
That is,
the curve $(x(s),y(s))$ satisfying \pref{lm-system} passes through the origin of $\mathbb{R}^{2}$. 
\begin{prop}
Suppose that $ \lim_{ s \to b} y(s) = 0 $ and $ \lim_{ s \to b } x(s) = 0 $ for some $ b \in I $.
Then,
there exists the limit of $ x^\prime (s)^2 $ as $ s \to b $,
and
\[
	\lim_{ s \to b } x^\prime (s)^2 = \frac l { l + m }.
\]
Hence, we have 
\begin{eqnarray*}
	\mbox{if} \ s < b, \mbox{ then} \lim_{ s \to b-0 } x^\prime (s)
	& \!\!\! = & \!\!\! - \sqrt{ \frac l { l + m } },
	\\
	\mbox{if} \ s > b, \mbox{ then} \lim_{ s \to b+0 } x^\prime (s)
	& \!\!\! = & \!\!\! \sqrt { \frac l { l + m } }.
\end{eqnarray*}
\label{prop.4.3}
\end{prop}
The proof of Proposition \ref{prop.4.3} is divided into several Lemmas.
\begin{lem}
If there exists $  \lim_{ s \to b } x^\prime (s)^{2} $,
then the formulas in Proposition \ref{prop.4.3} hold.
\label{lem.4.2}
\end{lem}
\begin{proof}
Set $\lim_{ s \to b } x^\prime (s)^{2} = X $.
By the second equation of \pref{lm-system},
there exists $\lim_{ s \to b } y^\prime (s)^2 =Y $.
By the same way as the proof of Lemma 4.1,
it is shown that $ x^\prime (s) $ and $ y^\prime (s)$ are not zero on a neighborhood of $ s = b $.
In a manner similar to the proof of Proposition 4.1,
we have
\[
	\lim_{ s \to b } x^\prime (s)
	=
	\frac 1m \lim_{ s \to b }
	\left\{
	( n - 1 ) H(s) y(s) + l \frac { y(s) } { x(s) } y^\prime (s)
	\right\}.
\]
Squaring both sides,
we obtain
\[
	X =
	\left( \frac 1m \right)^2
	\lim_{ s \to b }
	\left\{ ( n - 1 ) H(s) y(s) + l \frac { y(s) } { x(s) } y^\prime (s)
	\right\}^2
	=
	\left( \frac lm \right)^2
	\lim_{ s \to b } \left( \frac { y(s) } { x(s) } y^\prime (s) \right)^2.
\]
Suppose $ X \ne 0 $.
Since $ x(b) = y(b) = 0 $,
l'H\^{o}pital's rule leads to
\[
	\lim_{ s \to b } \left( \frac { y(s) } { x(s) } \right)^2
	=
	\lim_{ s \to b } \left( \frac { y^\prime (s) } { x^\prime (s) } \right)^2
	=
	\frac YX .
\]
Hence,
\[
	X = \left( \frac lm \right)^2 \frac { Y^2 } X.
\]
This with the second equation of \pref{lm-system} yields
\[
	X = \frac l { l+m } ,
	\quad
	Y = \frac m { l+m }.
\]
Last,
we need to show $ X \neq 0 $.
By contraries,
suppose $ X = 0 $.
Since $ Y =  1 $,
the same computation as above implies
\[
	0 = X = 
	\left( \frac lm \right)^2
	\lim_{ s \to b } \left( \frac { y(s) } { x(s) } y^\prime (s) \right)^2
	=
	\left( \frac lm \right)^2
	\lim_{ s \to b } \left( \frac { y^\prime (s) } { x^\prime (s) } \right)^2 Y
	= \infty ,
\]
giving a contradiction.
\qed
\end{proof}
\par
For simplicity we consider only the case $ s < b $.
\begin{lem}
It holds that
\[
	\liminf_{ s \to b - 0 }
	\frac { y^\prime (s) } { x^\prime (s) }
	=
	\frac ml \liminf_{ s \to b - 0 }
	\frac { x^\prime (s) } { y^\prime (s)} .
\]
\label{lem.4.3}
\end{lem}
\begin{proof}
If $ \lim_{ s \to b-0 }x^\prime (s)^2 $ exists,
then Lemma \ref{lem.4.3} follows from Lemma \ref{lem.4.2}.
In case $ \lim_{ s \to b-0 } x^\prime (s)^2 $ does not exist,
we will find a contradiction as follows.
When $ s $ is near to $ b $,
we have $ - 1 \leqq x^\prime (s) < 0 $.
Hence,
\[
	-1 \leqq
	\liminf_{ s \to b - 0 } x^\prime (s)
	<
	\limsup_{ s \to b - 0 } x^\prime (s)
	\leqq
	0.
\]
Then,
there is a sequence $ \{ s_j \} $  satisfying
\begin{eqnarray*}
	&&
	\lim_{ j \to \infty } s_j = b-0 ,
	\quad
	\lim_{ j \to \infty } x^\prime ( s_j )
	=
	\limsup_{ s \to b - 0 } x^\prime ( s ) > -1 , \\
	&&
	\lim_{ j \to \infty } y^\prime ( s_j )
	=
	\liminf_{ s \to b - 0 } y^\prime ( s ) < 0 , \,
	x^{ \prime \prime } ( s_j ) = y^{ \prime \prime } ( s_j ) = 0.
\end{eqnarray*}
By the first equation of \pref{lm-system},
we see that,
at $ s = s_j $, 
\[
	\frac { y ( s_j ) } { x ( s_j ) }
	=
	\frac ml \frac { x^\prime ( s_j ) } { y^\prime ( s_j ) }
	-
	\frac { ( n - 1 ) H( s_j ) y( s_j ) } { l y^\prime ( s_j ) }.
\]
When $ j \to \infty $,
since $ y^\prime ( s_j ) $ does not converge to $ 0 $,
the second term in the above right hand one tends to $0$.
Hence,
$ \lim_{ j \to \infty } \frac{ y ( s_j ) } { x ( s_j )} $ exists and we have
\[
	\lim_{ j \to \infty } \frac { y ( s_j ) } { x ( s_j ) }
	=
	\frac ml \lim_{ j \to \infty } \frac { x^\prime ( s_j ) } { y^\prime ( s_j ) } .
\]
By l'H\^{o}pital's rule for the inferior limit,
we have
\[
	\lim_{ j \to \infty } \frac { y ( s_j ) } { x ( s_j ) }
	\geqq
	\liminf_{ s \to b - 0 } \frac { y(s) } { x(s) }
	\geqq
	\liminf_{ s \to b - 0 } \frac { y^\prime (s) } { x^\prime (s) }.
\]
\par
Since $ \frac{ - \xi}{ \sqrt{ 1 - \xi^2 }} $ is monotone decreasing on the interval $ ( - 1 , 0 ] $,
it holds that
\[
	\lim_{ j \to \infty } \frac { x^\prime ( s_j ) } { y^\prime ( s_j ) }
	=
	\lim_{ j \to \infty }
	\left( - \frac { x^\prime ( s_j ) } { \sqrt{ 1 - x^\prime ( s_j )^2 } } \right)
	=
	\liminf_{ s \to b - 0 } \left( - \frac { x^\prime (s) } { \sqrt{ 1 - x^\prime (s)^2 } } \right)
	=
	\liminf_{ s \to b - 0 } \frac { x^\prime (s) } { y^\prime (s) }.
\]
Consequently,
we have
\[
	\liminf_{ s \to b - 0 } \frac { y^\prime (s) } { x^\prime (s) }
	\leqq
	\frac ml
	\liminf_{ s \to b - 0 } \frac { x^\prime (s) } { y^\prime (s) } .
\]
Next,
we show the opposite inequality.
There exists a sequence $ \{ \tilde s_j \} $ such that
\begin{eqnarray*}
	&&
	\lim_{ j \to \infty } \tilde s_j = b - 0 ,
	\quad
	\lim_{ j \to \infty } x^\prime ( \tilde s_j )
	=
	\liminf_{ s \to b - 0 } x^\prime (s) < 0 ,
	\\
	&&
	x^{ \prime \prime } ( \tilde s_j ) =
	y^{ \prime \prime } ( \tilde s_j ) = 0.
\end{eqnarray*}
By the second equation of \pref{lm-system},
\[
	\frac { y^\prime ( \tilde s_j ) } { x^\prime ( \tilde s_j ) }
	=
	\frac ml \frac { x( \tilde s_j ) } { y( \tilde s_j ) }
	-
	\frac { ( n - 1 ) H( \tilde s_j ) x ( \tilde s_j ) } { l x^\prime ( \tilde s_j ) }.
\]
As $ j \to \infty $,
we see
\[
	\lim_{ j \to \infty } \frac { y^\prime ( \tilde s_j ) } { x^\prime ( \tilde s_j ) }
	=
	\frac ml
	\lim_{ j \to \infty } \frac { x( \tilde s_j ) } { y( \tilde s_j ) }.
\]
Thus, we have
\begin{eqnarray*}
	&& \!\!\!\!\!\!\!\!\!\!\!\!
	\lim_{ j \to \infty }
	\frac { y^\prime ( \tilde s_j ) } { x^\prime ( \tilde s_j ) }
	=
	\lim_{ j \to \infty }
	\left( - \frac { \sqrt{ 1 - x^\prime ( \tilde s_j )^2 } } { x^\prime ( s_j ) } \right)
	=
	\liminf_{ s \to b - 0 }
	\left( - \frac { \sqrt{ 1 - x^\prime (s)^2 } } { x^\prime (s) }
	\right)
	=
	\liminf_{ s \to b - 0 }
	\frac { y^\prime (s) } { x^\prime (s) } ,
	\\
	&& \!\!\!\!\!\!\!\!\!\!\!\!
	\lim_{ j \to \infty } \frac { x( \tilde s_j ) } { y( \tilde s_j ) }
	\geqq
	\liminf_{ s \to b - 0 } \frac { x(s) } { y(s) }
	\geqq
	\liminf_{ s \to b - 0 } \frac { x^\prime (s) } { y^\prime (s) } ,
\end{eqnarray*}
so that,
we obtain
\[
	\liminf_{ s \to b - 0 }
	\frac { y^\prime (s) } { x^\prime (s) }
	\geqq
	\frac ml \liminf_{ s \to b - 0 }
	\frac { x^\prime (s) } { y^\prime (s)} ,
\]
proving Lemma \ref{lem.4.3}.
\qed
\end{proof}
\begin{lem}
It holds that
\[
	\limsup_{ s \to b - 0 }
	\frac { y^\prime (s) } { x^\prime (s) }
	=
	\frac ml \limsup_{ s \to b - 0 }
	\frac { x^\prime (s) } { y^\prime (s)}.
\]
\label{lem.4.4}
\end{lem}
\begin{proof}
By changing $ x $ and $ y $ in the previous computation in Lemma \ref{lem.4.3},
we see that
\[
	\liminf_{ s \to b - 0 }
	\frac { x^\prime (s) } { y^\prime (s) }
	=
	\frac lm \liminf_{ s \to b - 0 }
	\frac { y^\prime (s) } { x^\prime (s)}.
\]
By taking the inverse,
Lemma \ref{lem.4.4} is proved.
\qed
\end{proof}
\begin{lem} It holds that
\[
	\liminf_{ s \to b - 0 } \frac { x^\prime (s) } { y^\prime (s) }
	=
	\liminf_{ s \to b - 0 } \frac { x(s) } { y(s) }
	,
	\quad
	\limsup_{ s \to b - 0 } \frac { x^\prime (s) } { y^\prime (s) }
	=
	\limsup_{ s \to b - 0 } \frac { x(s) } { y(s) } .
\]
\label{lem.4.5}
\end{lem}
\begin{proof}
If $ \lim_{ s \to b - 0 } \frac{x^\prime (s)}{ y^\prime (s)} $ exists,
then Lemma \ref{lem.4.5} follows from l'H\^{o}pital's rule.
By contraries,
suppose that $ \lim_{ s \to b - 0 } \frac{x^\prime (s)}{ y^\prime (s)} $ does not exist.
By using the sequence $ \{ s_j \}$ used in the proof of Lemma \ref{lem.4.3},
we see that
\[
	\begin{array}{rl}
	\displaystyle{
	\liminf_{ s \to b - 0 } \frac { x^\prime (s) } { y^\prime (s) }
	}
	= & \!\!\!
	\displaystyle{
	\lim_{ j \to \infty } \frac { x^\prime ( s_j ) } { y^\prime ( s_j ) }
	=
	\frac lm \lim_{ j \to \infty } \frac { y( s_j ) } { x( s_j ) }
	}
	\\
	\geqq & \!\!\!
	\displaystyle{
	\frac lm \liminf_{ s \to b - 0 } \frac { y(s) } { x(s) }
	\geqq
	\frac lm \liminf_{ s \to b - 0 } \frac { y^\prime (s) } { x^\prime (s) }
	=
	\liminf_{ s \to b - 0 } \frac { x^\prime (s) } { y^\prime (s) }
	} .
	\end{array}
\]
Hence,
we have
\[
	\liminf_{ s \to b - 0 } \frac { y(s) } { x(s) }
	=
	\frac ml
	\liminf_{ s \to b - 0 } \frac { x^\prime (s) } { y^\prime (s) }
	=
	\liminf_{ s \to b - 0 } \frac { y^\prime (s) } { x^\prime (s) } .
\]
The formula about the superior limit in Lemma \ref{lem.4.3} is shown by changing $x$ and $y$ 
in the above formula, 
proving Lemma \ref{lem.4.5}.
\qed
\end{proof}
\par
\begin{lem}
There exist
$ \lim_{ s \to b-0 } \frac { y^\prime (s) } { x^\prime (s) } $ and $ \lim_{ s \to b-0 } x^\prime (s)^2 $. 
\label{lem.4.6}
\end{lem}
\begin{proof}
Set
\[
	A(s) = \frac { y(s) } { x(s) },
	\quad
	B(s) = \frac { y^\prime (s) } { x^\prime (s) } ,
	\quad
	\liminf_{ s \to b - 0 } A(s) = \underline L ,
	\quad
	\limsup_{ s \to b - 0 } A(s) = \bar L .
\]
If $ \underline L = \bar L $,
then the assertion follows from Lemma \ref{lem.4.5}.
\par
Assuming $ \underline L < \bar L $, we will find a contradiction.
By Lemma \ref{lem.4.5},
we have 
$ \liminf_{ s \to b - 0 } B(s) = \underline L $,
$ \limsup_{ s \to b - 0 } B(s) = \bar L $.
By Lemma \ref{lem.4.4},
we have  $ \bar L \underline L = \frac m l $. 
Hence, our assumption
implies  $ \displaystyle{ \underline L < \sqrt{ \frac ml } < \bar L } $.
Taking into consideration of the shape of the generating curve, we choose the 
the sequence $ \{ s_j \} $  such that 	the generating curve is tangent to the line $ y = L_j x $ at $ s = s_j $, 
$\lim_{ j \to \infty } A( s_j ) = \underline L $ and $\lim_{ j \to \infty } L_j = \underline L$ when 
$\lim_{ j \to \infty } s_j = b - 0 $. Next, we choose also the sequence $ \{ \tilde{s}_j \} $ such that
\[s_j < \tilde s_j < s_{ j+1 } < \tilde s_{ j+1 },  \lim_{ j \to \infty } A ( \tilde s_j ) = \bar L,
\mbox{and} \lim_{ j \to \infty } \tilde s_j = b - 0. \]

The last property above implies
\[
	B ( s_j ) = L_j \to \underline L \quad \mbox{as} \quad j \to \infty .
\]
Set
\[
	B_\epsilon = \{ ( A,B ) \in \mathbb{R}^2 \, | \, ( A - \underline L )^2 + ( B - \underline L )^2 < \epsilon^2 \} .
\]
If $ \epsilon > 0 $ is sufficiently small,
then we may assume that
\[
	( A( s_j ) , B( s_j ) ) \in B_\epsilon ,
	\quad
	( A( \tilde s_j ) , B( \tilde s_j ) ) \in B_\epsilon^c .
\]
Hence there exists $ \{ \hat s_j \} $ such that
\[
	s_j < \hat s_j < \tilde s_j ,
	\quad
	( A(s) , B(s) ) \in \bar B_\epsilon \quad \mbox{for} \quad s \in [ s_j , \hat s_j ) ,
	\quad
	( A( \hat s_j ) , B ( \hat s_j ) ) \in \partial B_\epsilon .
\]
In order to  consider the behavior of $ ( A(s) , B(s) ) $ on the 
interval $ I_j = [ s_j , \hat s_j ] $, we now compute
\[
	\begin{array}{rl}
	\displaystyle{
	\frac 12 \frac d { ds } \left( A(s) - \sqrt { \frac ml } \right)^2
	}
	= & \!\!\!
	\displaystyle{
	\left( A(s) - \sqrt { \frac ml } \right) A^\prime (s)
	}
	\\
	= & \!\!\!
	\displaystyle{
	\left( A(s) - \sqrt { \frac ml } \right) 
	\frac {( y^\prime (s) x(x) - y(s) x^\prime (s)) } { x(s)^2 }
	}
	\\
	= & \!\!\!
	\displaystyle{
	\frac { x^\prime (s) } { x(s) }
	\left( A(s) - \sqrt { \frac ml } \right) 
	( B(s) - A(s) ) .
	}
	\end{array}
\]
When $ s \in I_j $,
\[
	\left| A(s) - \sqrt { \frac ml } \right| \leqq C
	,
\]
\[
	| B(s) - A(s) |
	=
	| B(s) - \underline L - ( A(s) - \underline L ) |
	\leqq
	2 \epsilon .
\]
Therefore, it holds that
\[
	\left|
	\frac { x^\prime (s) } { x(s) }
	-
	\frac { y^\prime (s) } { y(s) }
	\right|
	=
	| A(s) - B(s) | \left| \frac { x^\prime (s) } { y(s) } \right|
	\leqq
	2 \epsilon \left| \frac { x^\prime (s) } { y(s) } \right| ,
\]
which implies
\[
	\frac { x^\prime (s) } { x(s) }
	=
	\frac { y^\prime (s) + O ( \epsilon ) x^\prime (s) } { y(s) }
	.
\]
Consequently, we have
\[
	\left|
	\frac 12 \frac d { ds } \left( A(s) - \sqrt { \frac ml } \right)^2
	\right|
	=
	\left|
	\frac { y^\prime (s) + O ( \epsilon ) x^\prime (s) } { y(s) }
	\right|
	O ( \epsilon )
	=
	\frac { O( \epsilon ) } { y(s) } ,
\]
where we used $ | x^\prime (s) | \leqq 1 $,
$ | y^\prime (s) | \leqq 1 $.
On the other hand
\[
	\begin{array}{rl}
	\displaystyle{
	\frac 12 \frac d { ds } \left( B(s) - \sqrt { \frac ml } \right)^2
	}
	= & \!\!\!
	\displaystyle{
	\left( B(s) - \sqrt { \frac ml } \right) B^\prime (s)
	}
	\\
	= & \!\!\!
	\displaystyle{
	\left( B(s) - \sqrt { \frac ml } \right)
	\frac { (y^{\prime\prime} (s) x^\prime (s) - y^\prime (s) x^{ \prime \prime } (s) )} { ( x^\prime (s) )^2 }
	}
	\\
	= & \!\!\!
	\displaystyle{
	- \frac 1 { ( x^\prime (s) )^2 }
	\left( B(s) - \sqrt { \frac ml } \right)
	\left\{
	l \frac { y^\prime (s) } { x(s) } - m \frac { x^\prime (s) } { y(s) }
	+
	( n - 1 ) H(s)
	\right\}
	}
	\\
	= & \!\!\!
	\displaystyle{
	- \frac l { x^\prime (s) y(s) }
	\left( B(s) - \sqrt { \frac ml } \right)
	\left( A(s) B(s) - \frac ml \right)
	}
	\\
	& \quad
	\displaystyle{
	- \,
	\frac { ( n - 1 ) H(s) } { ( x^\prime (s) )^2 }
	\left( B(s) - \sqrt { \frac ml } \right) .
	}
	\end{array}
\]
Set $ \displaystyle{ \underline L = \lambda \sqrt{ \frac ml } } $,
and then $ 0 \leqq \lambda < 1 $.
We have on $ I_j $,
\[
	0 < A(s) < \frac { 1 + \lambda } 2 \sqrt{ \frac ml } ,
	\quad
	0 < B(s) < \frac { 1 + \lambda } 2 \sqrt{ \frac ml } ,
	\quad
	-1 \leqq x^\prime (s) < 0 ,
	\quad
	y(s) > 0
\]
for large $ j $.
Hence there exists $ \delta > 0 $ independent of $ \epsilon $ such that
\[
	- \frac l { x^\prime (s) y(s) }
	\left( B(s) - \sqrt { \frac ml } \right)
	\left( A(s) B(s) - \frac ml \right)
	\geqq
	\frac \delta { y(s) } .
\]
Since it holds on the interval $ I_j $ that
\[
	\frac { 1 - ( x^\prime (s) )^2 } { ( x^\prime (s) )^2 }
	=
	\left( \frac { y^\prime (s) } { x^\prime (s) } \right)^2
	=
	\underline L^2 ( 1 + O(1) )
	=
	O(1) ,
\]
we have
\[
	\inf \left\{ ( x^\prime (s) )^2 \, \left| \,
	s \in \bigcup_j I_j \right. \right\} > 0 .
\]
Hence
\[
	\left| \frac { ( n -1 ) H(s) } { ( x^\prime (s) )^2 }
	\left( B(s) - \sqrt{ \frac ml } \right) \right|
	\leqq C .
\]
Consequently, we have
\[
	\frac 12 \frac d { ds }
	\left\{
	\left( A(s) - \sqrt{ \frac ml } \right)^2
	+
	\left( B(s) - \sqrt{ \frac ml } \right)^2
	\right\}
	\geqq
	\frac 1 { y(s) }
	\left( \delta + O( \epsilon ) \right)
	- C
\]
on $ I_j $.
If $ j $ is sufficiently large,
then $ y(s) > 0 $ is sufficiently small.
Taking $ \epsilon $ small,
we have
\[
	\frac 12 \frac d { ds }
	\left\{
	\left( A(s) - \sqrt{ \frac ml } \right)^2
	+
	\left( B(s) - \sqrt{ \frac ml } \right)^2
	\right\}
	\geqq
	\frac \delta { 2 y(s) } > 0
\]
on $ I_j $ for large $ j $.
Hence
\[
	\begin{array}{rl}
	\displaystyle{
	\left( A( \hat s_j ) - \sqrt{ \frac ml } \right)^2
	+
	\left( B( \hat s_j ) - \sqrt{ \frac ml } \right)^2
	}
	\geqq & \!\!\!
	\displaystyle{
	\left( A( s_j ) - \sqrt{ \frac ml } \right)^2
	+
	\left( B( s_j ) - \sqrt{ \frac ml } \right)^2
	}
	\\
    = & \!\!\!
	\displaystyle{
	2 \left( L_j - \sqrt{ \frac ml } \right)^2 ,
	}
	\end{array}
\]
where the equality follows from the fact $A(s_{j}) =B(s_{j})=L_{j}$.
Taking a suitable subsequence,
we have $ ( A( \hat s_j ) , B( \hat s_j ) ) \to ( \hat A , \hat B ) $ as $ j \to \infty $,
where
\[
	( \hat A , \hat B )
	\in
	\partial B_\epsilon
	\cap
	\left\{
	( A , B ) \in \mathbb{R}^2 \, \left| \,
	\left( A - \sqrt{ \frac ml } \right)^2
	+
	\left( B - \sqrt{ \frac ml } \right)^2
	\geqq
	2 \left( \underline L - \sqrt{ \frac ml } \right)^2 \right. \right\} .
\]
This shows that
\[
	\hat A < \underline L
	\quad \mbox{or} \quad
	\hat B < \underline L .
\]
This is a contradiction.
Indeed,
if $ \hat A < \underline L $,
then
\[
	\liminf_{ s \to b - 0 } A(s) = \underline L
	> \hat A
	= \lim_{ j \to \infty } A( \hat s_j )
	\geqq
	\liminf_{ s \to b - 0 } A(s) .
\]
If $ \hat B < \underline L $,
then
\[
	\liminf_{ s \to b - 0 } B(s) = \underline L
	> \hat B
	= \lim_{ j \to \infty } B( \hat s_j )
	\geqq
	\liminf_{ s \to b - 0 } B(s) .
\]
\par
Hence we have  $ \underline L = \bar L $, proving Lemma \ref{lem.4.6}.
\qed
\end{proof}
\par
Proposition \ref{prop.4.3} is proved by Lemmas \ref{lem.4.2}--\ref{lem.4.6}.

We are now in a position to study the system \pref{DiffEq2} with $ x(0) = 0 $ and $ y(0) = 0 $.
Proposition \ref{prop.4.3} tells us $ q(0)^2 = \frac lm $.
We may suppose $q(0)= \sqrt{\frac{l}{m}}$,
because the generating curve is in the domain of $ x > 0 $ and $ y > 0 $.
Thus,
in this case,
we have the following singular initial value problem
\begin{equation}
	\left\{
	\begin{array}{rl}
	\displaystyle{
	y \frac { dq } { dy }
	}
	= & \!\!\!
	\displaystyle{
	( 1 + q^2 )
	\left(
	- mq + \frac { ly } { \int_0^y q( \xi ) \, d \xi }
	\right)
	+ ( n - 1 ) \tilde H (y) ( 1 + q^2 )^{ \frac 32 } y,
	}
	\\
	q(0) = & \!\!\!
	\displaystyle{
	\sqrt{ \frac lm }.
	}
	\end{array}
	\right.
	\label{DiffEq-qlm}
\end{equation}
We furthermore transform this  by setting
\[
	q(y) = \sqrt{ \frac lm } + r(y) .
\]
The new one for $ r(y) $ is obtained in the following Lemma.
\begin{lem}
$ r = r(y) $ satisfies
\begin{equation}
	\left\{
	\begin{array}{rl}
	\displaystyle{
	y \frac { dr } { dy }
	}
	= & \!\!\!
	\displaystyle{
	- ( l + m ) r(y) + F_1 (r(y)) + F_2 ( r( \cdot ), y ) + F_3 (r(y),y),
	}
	\\
	r(0) = & \!\!\! 0 ,
	\end{array}
	\right. 
	\label{DiffEq-r}
\end{equation}
where we set
\begin{eqnarray*}
	F_1 ( r(y) ) &=& - r(y)^2 \left( m r(y) + 2 \sqrt{ lm } \right),  \\
	F_2 ( r( \cdot ) ,y ) &=&
	- \sqrt{ lm }
	\left\{ 1 + \left( r(y) + \sqrt{ \frac lm } \right)^2 \right\}
	\frac { \displaystyle{
	\frac { \sqrt m } y \int_0^y r( \xi ) \, d \xi
	} }
	{ \displaystyle{
	\frac { \sqrt m } y \int_0^y r( \xi ) \, d \xi + \sqrt l
	} },  \\
	F_3 ( r(y) , y )
	&=&
	( n - 1 ) \tilde H (y)
	\left\{ 1 + \left( r(y) + \sqrt{ \frac lm } \right)^2 \right\}^{ \frac 32 } y.
\end{eqnarray*}
\label{lem.4.7}
\end{lem}
\begin{proof}
For the right hand side of the first equation of \pref{DiffEq-qlm},
we compute
\begin{eqnarray*}
	&&
	\displaystyle{
	( 1 + q^2 )
	\left(
	- mq + \frac { ly } { \displaystyle{ \int_0^y q( \xi ) \, d \xi } }
	\right)
	+
	( n - 1 ) \tilde H(y) ( 1 + q^2 )^{ \frac 32 } y
	}
	\\
	&&
	=
	\displaystyle{
	\left\{ 1 + \left( \sqrt{ \frac lm } + r \right)^2 \right\}
	\left\{
	- m \left( \sqrt { \frac lm } + r \right)
	+ \frac { ly } { \displaystyle{ \int_0^y \left( \sqrt{ \frac lm } + r ( \xi ) \right) d \xi } }
	\right\}
	}
	\\
	&&
	\qquad \qquad
	\displaystyle{
	+ \,
	( n - 1 ) \tilde H(y)
	\left\{ 1 + \left( \sqrt{ \frac lm } + r \right)^2 \right\}^{ \frac 32 } y
	}
	\\
	\quad
	&&
	=
	\displaystyle{
	-
	\left\{ 1 + \left( \sqrt{ \frac lm } + r \right)^2 \right\}
	mr
	}
	\\
	&&
	\qquad \qquad
	\displaystyle{
	+ \,
	\left\{ 1 + \left( \sqrt{ \frac lm } + r \right)^2 \right\}
	\left \{
	- m \sqrt{ \frac lm }
	+ \frac { ly }
	{ \displaystyle{ \int_0^y \left( \sqrt{ \frac lm } + r ( \xi ) \right) d \xi } }
	\right\}
	}
	\\
	&&
	\qquad \qquad
	\displaystyle{
	+ \,
	( n - 1 ) \tilde H(y)
	\left\{ 1 + \left( \sqrt{ \frac lm } + r \right)^2 \right\}^{ \frac 32 } y.
	}
\end{eqnarray*}
Since we have
\[
	\begin{array}{l}
	\displaystyle{
	-
	\left\{ 1 + \left( \sqrt{ \frac lm } + r \right)^2 \right\}
	mr
	}
	\\
	\quad
	=
	\displaystyle{
	- \left( 1 + \frac lm + 2 \sqrt{ \frac lm } r + r^2 \right) mr
	=
	- \left( m + l + 2 \sqrt{ lm } r + m r^2 \right) r
	}
	\\
	\quad
	=
	\displaystyle{
	- ( l + m ) r - r^2 \left( mr + 2 \sqrt{ lm } \right),
	}
	\end{array}
\]
and
\[
	\begin{array}{rl}
	\displaystyle{
	- m \sqrt{ \frac lm }
	+ \frac { ly }
	{ \displaystyle{ \int_0^y \left( \sqrt{ \frac lm } + r ( \xi ) \right) d \xi } }
	}
	=
	& \!\!\!
	\displaystyle{
	- \sqrt{ lm }
	+ \frac { l \sqrt m }
	{ \displaystyle{ \sqrt l + \frac { \sqrt m } y \int_0^y r( \xi )\, d \xi } }
	}
	\\
	\quad
	=
	& \!\!\!
	\displaystyle{
	- \frac { \displaystyle{ \frac { m \sqrt l } y \int_0^y r( \xi ) \, d \xi } }
	{ \displaystyle{ \sqrt l + \frac { \sqrt m } y \int_0^y r( \xi )\, d \xi } },
	}
	\end{array}
\]
Lemma \ref{lem.4.7} was proved.
\qed
\end{proof}
\begin{rem}
There is a non-local part $ \int_0^y r( \xi ) \, d \xi $ in $ F_2 $.
Therefore we should write $ F_2 ( r ( \cdot ) , y ) $,
not $ F ( r(y) , y ) $.
That is,
$ F_2 $ is defined on $ ( \mbox{a function space} ) \times \mathbb{R} $,
not on $ \mathbb{R}^2 $.
\end{rem}
\par
Multiplying $ y^{l+m-1} $ on the first equation of \pref{DiffEq-r} and integrating it with respect to $y$, we have
an integral equation 
\begin{equation}
	r(y) = 	\Theta (r) (y) ,
	\label{IntEq-r}
\end{equation}
where we set
\begin{eqnarray*}
	\Theta (r) (y)
	& \!\!\! = & \!\!\!
	\frac 1 { y^{ l+m } } \int_0^y F( \eta ) \eta^{ l + m-1 } d \eta ,
	\\
	F(y)
	& \!\!\! = & \!\!\! F_1 (r(y)) + F_2 (r( \cdot),y) + F_3 (r(y),y).
\end{eqnarray*}
Using the same notations $X_Y$ and $X_{Y,M}$ defined in \S~2,
we have
\begin{prop}
\begin{itemize}
\item[{\rm (i)}]
If $ \tilde H $ is bounded,
then there exist constants $M$ and $Y$ such that the integral equation {\rm \pref{IntEq-r}} has a unique solution $r$.
\item[{\rm (ii)}]
If $ \tilde H $ is bounded and continuous,
then the solution $r$ given in {\rm (i)} is a solution of {\rm \pref{DiffEq-r}}.
\end{itemize}
\label{prop.4.4}
\end{prop}
\begin{proof}
We shall find a fixed point of the mapping $ \Theta : X_{Y,M} \longrightarrow  X_{Y,M}$.
The proof is accomplished by the following two steps.
\par
First we show that there exist $M$ and $Y$ such that $\Theta$ is a mapping from $ X_{Y,M} $ into itself.
Take any $ r \in X_{Y,M} $.
Then we have
\[
	| F_1 (r(y)) |
	\leqq
	M^2 y^2
	\left( mM y + 2 \sqrt{ lm } \right)
	\leqq
	C \left( M^2 y^2 + M^3 y^3 \right).
\]
By virtue of
\[
	\left|
	\frac 1y \int_0^y r( \xi ) \, d \xi
	\right|
	\leqq
	\frac 1y \int_0^y \| r \|_X \xi \, d \xi
	\leqq
	\frac { M y } 2,
\]
if constants $M$ and $Y$ satisfy $ \displaystyle{ MY < 2 \sqrt { \frac lm } } $,
then we have
\[
	\sqrt{ \frac ml } \frac 1y \int_0^y r( \xi ) \, d \xi + 1
	\geqq
	1 - \sqrt{ \frac ml } \frac { My } 2
	> C^{-1} > 0 .
\]
Then,
\[
	\begin{array}{rl}
	| F_2 (r( \cdot ),y) |
	\leqq & \!\!\!
	\displaystyle{
	\sqrt{ lm }
	\left\{ 1 + \left( My + \sqrt{ \frac lm } \right)^2 \right\}
	\frac { \displaystyle{ \frac { \sqrt m } 2 My } }
	{ \displaystyle{ \sqrt l - \frac { \sqrt m } 2 M y } }
	}
	\\
	\leqq & \!\!\!\!
	\displaystyle{
	\frac { \sqrt l } { \displaystyle{ \sqrt l - \frac { \sqrt m } 2 M y } }
	\frac { l + m } 2 M y
	+ C \left( M^2 y^2 + M^3 y^3 \right).
	}
	\end{array}
\]
Also, by virtue of the inequality
\[
	\left\{1 + \left( My + \sqrt {\frac{l}{m}} \right)^2 \right\}^{\frac{1}{2}}
	\leqq 1 + My + \sqrt {\frac{l}{m}}
	\quad \mbox{for} \quad y > 0 ,
\]
we have
\[
	| F_3 (r(y),y) |
	\leqq
	C \left( 1 + M^3 y^3 \right) y
	=
	C \left( y + M^3 y^4 \right).
\]
Suppose that $M$ and $Y$ satisfy $ \sqrt m MY \leqq \sqrt l $, which implies   
\[
	\frac{\sqrt{l}}{ \displaystyle{ \sqrt{l} - \frac{\sqrt{m}}{2} M \eta } } \leqq 2
	\quad \mbox{for} \quad \eta \in (0,Y] .
\]
Then,
\[
	\begin{array}{rl}
	\displaystyle{\left| \frac { \Theta (r) (y) } {y} \right|}
	\leqq & \!\!\!
	\displaystyle{
	\frac 1 { y^{ l+m+1 } } \int_0^y
	\left( | F_1 (r(\eta)) | + | F_2 (r(\cdot),\eta) | + | F_3 (r(\eta),\eta) | \right) \eta^{ l+m-1 } d \eta
	} \\
	\leqq & \!\!\!
	\displaystyle{
	\frac 1 { y^{ l+m+1 } } \int_0^y
	\left\{
	\frac { \sqrt l } { \displaystyle{ \sqrt l - \frac { \sqrt m } {2} M \eta } }
	\frac{l + m}{2} M \eta \right. } \\
    &
	\qquad \qquad
	\displaystyle{
	\left.
	+ \,
	C \left( \eta + M^2 \eta^2 + M^3 \eta^3 + M^3 \eta^4 \right)
	\vphantom{ \frac { \sqrt 1 } { \displaystyle{ \frac { \sqrt m } 2 } } }
	\right\}
	\eta^{ l+m-1 } d \eta
	}
	\\
	\leqq & \!\!\!
	\displaystyle{
	\frac { l+m } { l+m+1 } M
	+
	C \left( 1 + M^2 Y + M^3 Y^2 + M^3 Y^3 \right)}.
	\end{array}
\]
We take some $M$ and $Y$ satisfying these two conditions:
\begin{eqnarray}
&&	MY < \sqrt { \frac lm } ,
	\label{MY1}    \\
&&	\frac { l+m } { l+m+1 } M
	+
	C \left( 1 + M^2 Y + M^3 Y^2 + M^3 Y^3 \right)
	\leqq M.
	\label{MY2}
\end{eqnarray}
Then,
$ \Theta $ maps $ X_{M,Y} $ into itself.
\par
Suppose that $ M $ and $ Y $ satisfy (\ref{MY1}) and (\ref{MY2}).
Then,
we note that any $M$ and $Y'(<Y)$ also satisfy the same conditions (\ref{MY1}) and (\ref{MY2}).
Next we show that,
taking $ Y $  more smaller if necessary,
$ \Theta $ is a contraction mapping from $ X_{Y,M} $ into itself.
Take $ r_1 $,
$ r_2 \in X_{Y,M} $.
Then we have
\[
	\Theta ( r_1 ) (y) - \Theta ( r_2 ) (y)
	=
	\frac 1 { y^{ l+m } }
	\int_0^y
	\sum_{ j=1 }^3 \left( F_{j1} (\eta) - F_{j2} (\eta) \right) \eta^{ l+m-1 } d \eta ,
\]
where
\[
	F_{1k} (y) = F_1 ( r_k (y) ),
	\quad
	F_{2k} (y) = F_2 ( r_k ( \cdot ) , y ) ,
	\quad
	F_{3k} (y) = F_3 ( r_k (y) , y )\quad (k=1,2).
\]
For $F_{1k}(y)$,
we see that
\[
	\begin{array}{l}
	\displaystyle{
	| F_{11} (y) - F_{12} (y) |
	}
	\\
	\quad
	=
	\displaystyle{
	\left|
	\left\{
	m \left( r_1(y)^2 + r_1 (y) r_2 (y) + r_2 (y)^2 \right)
	+ 2\sqrt{lm } \left( r_1 (y) + r_2 (y) \right)
	\right\}
	( r_1 (y) - r_2 (y) )
	\right|
	}
	\\
	\quad
	\leqq
	\displaystyle{
	C \left( My^2 + M^2 y^3 \right) \| r_1 - r_2 \|_X .
	}
	\end{array}
\]
For $F_{2k}(y)$,
we set
$$
	F_{21}(y) - F_{22}(y) = G_{1}(y) + G_{2}(y) + G_{3}(y),
$$
where
\begin{eqnarray*}
	G_{1}(y)
	& \!\!\! = & \!\!\!
	\displaystyle{
	- \sqrt{lm}
	\left\{
	\left( r_1(y) + \sqrt{ \frac lm } \right)^2
	-
	\left( r_2(y) + \sqrt{ \frac lm } \right)^2
	\right\}
	\frac
	{ \displaystyle{ \frac { \sqrt m } y \int_0^y r_1( \xi ) \, d \xi } }
	{ \displaystyle{ \frac { \sqrt m } y \int_0^y r_1( \xi )\, d \xi + \sqrt l } } ,
	}
	\\
	G_{2}(y)
	& \!\!\! = & \!\!\!
	\displaystyle{
	- \sqrt{ lm }
	\left\{
	1 + \left( r_2 (y) + \sqrt { \frac lm } \right)^2
	\right\}
	\frac
	{ \displaystyle{ \frac { \sqrt m } y \int_0^y\left( r_1( \xi ) - r_2 ( \xi ) \right) \, d \xi } }
	{ \displaystyle{ \frac { \sqrt m } y \int_0^y r_1( \xi )\, d \xi + \sqrt l } }
	}
	,
	\\
	G_{3}(y)
	& \!\!\! = & \!\!\!
	\displaystyle{
	\sqrt{ lm }
	\left\{
	1 + \left( r_2 (y) + \sqrt { \frac lm } \right)^2
	\right\}
	\frac { \sqrt m } y \int_0^y r_2( \xi ) \, d \xi}  \\
	&&
	\qquad \quad
	\displaystyle{
	\times
	\frac
	{ \displaystyle{ \frac { \sqrt m } y \int_0^y ( r_1 ( \xi ) - r_2 ( \xi ) ) \, d \xi } }
	{ \displaystyle{
	\left( \frac { \sqrt m } y \int_0^y r_1( \xi )\, d \xi + \sqrt l \right)
	\left( \frac { \sqrt m } y \int_0^y r_2( \xi )\, d \xi + \sqrt l \right) ,
	} }
	}
\end{eqnarray*}
and then we have
\[
	\begin{array}{rl}
	| G_1 (y) |
	= & \!\!\!
	\displaystyle{
	\left|
	\sqrt{lm}
	\left( r_1 (y) + r_2 (y)+ 2 \sqrt{ \frac lm } \right)
	( r_1 (y) - r_2 (y) )
	\frac
	{ \displaystyle{ \frac { \sqrt m } y \int_0^y r_1( \xi ) \, d \xi } }
	{ \displaystyle{ \frac { \sqrt m } y \int_0^y r_1( \xi )\, d \xi + \sqrt l } }
	\right|
	}
	\\
	\leqq & \!\!\!
	\displaystyle{
	2 \sqrt{ lm } \left( My + \sqrt { \frac lm } \right)
	\| r_1 - r_2 \|_X y
	\frac
	{ \displaystyle{ \frac { \sqrt m } 2 M y } }
	{ \displaystyle{ \sqrt l - \frac { \sqrt m } 2 My } }
	}
	\\
	\leqq & \!\!\!
	\displaystyle{
	2 \left( m M y + \sqrt{ lm } \right) M y^2 \| r_1 - r_2 \|_X
	}
	\\
	\leqq & \!\!\!
	\displaystyle{
	C \left( M^2 y^3 + M y^2 \right) \| r_1 - r_2 \|_X
	}
	,
	\end{array}
\]
\[
	\begin{array}{rl}
	| G_2 (y) |
	\leqq & \!\!\!
	\displaystyle{
	\sqrt{lm}
	\left\{ 1 + \left( My + \sqrt{ \frac lm } \right)^2 \right\}
	\frac
	{ \displaystyle{ \frac { \sqrt m } y \int_0^y \| r_1 - r_2 \|_X \xi \, d \xi } }
	{ \displaystyle{ \sqrt l - \frac { \sqrt m } 2 My } }
	}
	\\
	\leqq & \!\!\!
	\displaystyle{
	\left( m + m M^2 y^2 + 2 \sqrt{ lm } My + l \right) y \| r_1 - r_2 \|_X
	}
	\\
	\leqq & \!\!\!
	\displaystyle{
	\left\{
	C \left( M^2 y^3 + M y^2 \right)
	+ ( l + m ) y
	\right\}
	\| r_1 - r_2 \|_X
	}
	,
	\end{array}
\]
and
\[
	\begin{array}{rl}
	| G_3 ( y ) |
	\leqq & \!\!\!
	\displaystyle{
	\sqrt { lm }
	\left\{
	1 + \left( My + \sqrt{ \frac lm } \right)^2
	\right\}
	\frac { \sqrt m } 2 M y
	\frac
	{ \displaystyle{ \frac { \sqrt m } 2 y \| r_1 - r_2 \|_X } }
	{ \displaystyle{ \left( \sqrt l - \frac { \sqrt m } 2 M y \right)^2 } }
	}
	\\
	\leqq & \!\!\!
	\displaystyle{
	\sqrt{ \frac ml }
	\left( m + m M^2 y^2 + 2 \sqrt{ lm } My + l \right)
	My^2 \| r_1 - r_2 \|_X
	}
	\\
	\leqq & \!\!\!
	\displaystyle{
	C \left( M y^2 + M^3 y^4 + M^2 y^3 \right)
	\| r_1 - r_2 \|_X
	}
	.
	\end{array} 
\]
Therefore it holds that
\[
	| F_{21} (y) - F_{22} (y) |
	\leqq
	\left\{
	C \left( M y^2 + M^3 y^4 \right) + ( l + m ) y
	\right\}
	\| r_1 - r_2 \|_X .
\]
Last,
we estimate $F_{3k}(y)$.
By the mean value theorem,
for each $ y \in [ 0 , Y ] $,
there exists $ r_\ast (y) $ between $ r_1 (y)$ and $ r_2 (y) $,
such that
\[
	\begin{array}{l}
	\displaystyle{
	\left\{
	1 + \left( r_1(y) + \sqrt{ \frac lm } \right)^2
	\right\}^{ \frac 32 }
	-
	\left\{
	1 + \left( r_2(y) + \sqrt{ \frac lm } \right)^2
	\right\}^{ \frac 32 }
	}
	\\
	\quad
	=
	\displaystyle{
	3 \left( r_\ast (y) + \sqrt{ \frac lm } \right)
	\left\{ 1 + \left( r_\ast (y) + \sqrt{ \frac lm } \right)^2 \right\}^{ \frac 12 }
	( r_1(y) - r_2 (y) )
	}.
	\end{array}
\]
Hence,
we have
\[
	\begin{array}{l}
	| F_{31} (y) - F_{32} (y) |
	\\
	\quad =
	\displaystyle{
	\left|
	( n - 1 ) \tilde H (y) y
	\left[
	\left\{
	1 + \left( r_1(y) + \sqrt{ \frac lm } \right)^2
	\right\}^{ \frac 32 }
	-
	\left\{
	1 + \left( r_2(y) + \sqrt{ \frac lm } \right)^2
	\right\}^{ \frac 32 }
	\right]
	\right|
	}
	\\
	\quad
	\leqq
	\displaystyle{
	3 ( n - 1 ) y \sup_{ \eta \in [ 0 , Y ] } | \tilde H( \eta ) |
	\left( My + \sqrt{ \frac lm } \right)
	\left\{ 1 + \left( My + \sqrt{ \frac lm } \right)^2 \right\}^{ \frac 12 }
	y \| r_1 - r_2 \|_X
	}
	\\
	\quad
	\leqq
	\displaystyle{
	C \left( 1 + M^2 y^2 \right) y^2 \| r_1 - r_2 \|_X
	}.
	\end{array}
\]
Consequently,
we see that
\[
	\begin{array}{l}
	\displaystyle{
	\left|
	\frac { \Theta ( r_1 ) (y) - \Theta ( r_2 ) (y) } y
	\right|
	}
	\\
	\quad
	\leqq
	\displaystyle{
	\frac { \| r_1 - r_2 \|_X } { y^{ l+m+1 } }
	\int_0^y
	\left\{
	C \left(
	M \eta^2 + M^3 \eta^4
	+ \eta^2 + M^2 \eta^4
	\right)
	+ ( l + m ) \eta
	\right\}
	\eta^{ l + m - 1 }
	d \eta
	}
	\\
	\quad
	\leqq
	\displaystyle{
	\left\{
	C \left(
	M Y + M^3 Y^3
	+ Y
	+ M^2 Y^3
	\right)
	+
	\frac { l + m } { l + m + 1 }
	\right\}
	\| r_1 - r_2 \|_X
	}.
	\end{array}
\]
Let us choose  $Y$ by
\[
	C \left(
	M Y + M^3 Y^3
	+ Y
	+ M^2 Y^3
	\right)
	+
	\frac { l + m } { l + m + 1 }
	< 1 .
\]
Then,
$ \Theta $ is a contraction mapping from $ X_{Y,M} $ to itself.
\par
By Banach's fixed point theorem,
there exists uniquely a fixed point $ r $ of $ \Theta $ on $ X_{Y,M} $. 
This $ r $ is a solution of the integral equation (\ref{IntEq-r}).
If $ \tilde H $ is continuous, 
then $ r $ satisfies the first equation in (\ref{DiffEq-r}).
Since $ r \in X_{Y} $,
it satisfies also the second equation of (\ref{DiffEq-r}),
proving Proposition \ref{prop.4.4}.
\qed
\end{proof}
\par
We note that $ q(y) = \sqrt { \frac lm } + r(y) $ is a solution of
(\ref{DiffEq-qlm}).
By Propositions \ref{prop.4.3} and \ref{prop.4.4},
we finished the proof of the case (b).

By replacing \cite[Proposition 3.3]{doke2} to Propositions \ref{prop.4.3} and \ref{prop.4.4} of this paper, we prove the 
following Theorem  in the  same way as the  proof of \cite[Theorem 3.4]{doke2}.
\begin{thm}
Let $H(s)$ be a continuous function on $\mathbb{R}$,
and fix an $s_{0} \in \mathbb{R}$.
Then,
for any positive numbers 
$c>0$,
$d>0$ and any real numbers $c^\prime $,
$ d^\prime$ with ${c^\prime}^2 + {d^\prime}^{2} = 1$,
there exists a global solution curve $(x(s),y(s))$,
$ s \in \mathbb{R}$,
of {\rm\pref{lm-system}} with the initial conditions $x(s_{0})=c$,
$ y(s_{0})=d$,
$ x^\prime(s_{0})=c^\prime$,
$ y^\prime(s_{0})=d^\prime$.
\label{thm.4.1}
\end{thm}

\begin{rem}
In case of $ s \to b+0 $,
$ l $ principal curvatures $ - \frac { y^\prime (s) } { x(s) } $ tends to $ - \infty $,
and $ m $ principal curvatures $ \frac { x^\prime (s) } { y(s) } $ to $ \infty $,
but the sum of all principal curvatures
\[
	- l \frac { y^\prime (s) } { x(s) }
	+m \frac { x^\prime (s) } { y(s) }
	+ x^{\prime \prime} (s) y^\prime (s) - x^\prime (s) y^{\prime\prime} (s)
\]
 remains bounded and tends to $ ( n - 1 ) H(b) $.
This suggests that
for the generalized rotational hypersurface  $ M $
 of $O(l+1)\times O(m+1)$-type, the asymptotic shape of $ x (s) S^l $- part as $ s \to b+0 $ is the negatively 
 curved $ l + 1 $ dimensional cone with  at $ s = b $,
and that of $ y(s) S^m $- part is the positively curved $ m + 1 $ dimensional cone.
When $ s \to b-0 $,
the asymptotic shape is a similar one with reverse orientation.
\end{rem}
\begin{center}
\begin{tabular}{ll}
Katsuei Kenmotsu: &  Takeyuki Nagasawa:  \\
Mathematical Institute & Department of Mathematics  \\
                       & Graduate School of Science and Engineering  \\
T\^{o}hoku University & Saitama University  \\
Sendai 980--8578  & Saitama 338--8570  \\
Japan  & Japan  \\
\vspace{0.2cm}
kenmotsu@math.tohoku.ac.jp & tnagasaw@rimath.saitama-u.ac.jp
\end{tabular}
\end{center}
\end{document}